\documentclass[a4paper,11pt]{article}
\title{\bf{Limit stable objects on Calabi-Yau 3-folds}
}
\date{}
\author{Yukinobu Toda}

\usepackage{makeidx}

\usepackage{latexsym}
\usepackage{amscd}
\usepackage{amsmath}
\usepackage{amssymb}
\usepackage{amsthm}
\usepackage{float}
\usepackage[all]{xy}
\usepackage{graphicx}

\usepackage{array}
\usepackage{amscd}
\usepackage[all]{xy}
\usepackage{makeidx}
\usepackage{latexsym}
\DeclareFontFamily{U}{rsfs}{%
\skewchar\font127}
\DeclareFontShape{U}{rsfs}{m}{n}{%
<-6>rsfs5<6-8.5>rsfs7<8.5->rsfs10}{}
\DeclareSymbolFont{rsfs}{U}{rsfs}{m}{n}
\DeclareSymbolFontAlphabet
{\mathrsfs}{rsfs}
\DeclareRobustCommand*\rsfs{%
\@fontswitch\relax\mathrsfs}
\newdimen\argwidth
\def\db[#1\db]{
 \setbox0=\hbox{$#1$}\argwidth=\wd0
 \setbox0=\hbox{$\left(\box0\right)$}
  \advance\argwidth by -\wd0
 \left(\kern.3\argwidth\box0 \kern.3\argwidth\right)}

\theoremstyle{plain}
\newtheorem{thm}{Theorem}[section]
\newtheorem{prop}[thm]{Proposition}
\newtheorem{lem}[thm]{Lemma}

\newtheorem{defi}[thm]{Definition}
\newtheorem{rmk}[thm]{Remark}

\newtheorem{prop-defi}[thm]{Proposition-Definition}
\newtheorem{lem-defi}[thm]{Lemma-Definition}

\newtheorem{conj}[thm]{Conjecture}
\newtheorem{exam}[thm]{Example}

\newcommand{\aA}{\mathcal{A}}

\newcommand{\cC}{\mathcal{C}}
\newcommand{\dD}{\mathcal{D}}
\newcommand{\eE}{\mathcal{E}}
\newcommand{\fF}{\mathcal{F}}
\newcommand{\gG}{\mathcal{G}}
\newcommand{\hH}{\mathcal{H}}

\newcommand{\lL}{\mathcal{L}}
\newcommand{\mM}{\mathcal{M}}
\newcommand{\nN}{\mathcal{N}}
\newcommand{\oO}{\mathcal{O}}
\newcommand{\pP}{\mathcal{P}}

\newcommand{\sS}{\mathcal{S}}
\newcommand{\tT}{\mathcal{T}}
\newcommand{\uU}{\mathcal{U}}

\newcommand{\lr}{\longrightarrow}

\newcommand{\Supp}{\mathop{\rm Supp}\nolimits}
\newcommand{\Hom}{\mathop{\rm Hom}\nolimits}

\newcommand{\dR}{\mathbf{R}}

\newcommand{\Hilb}{\mathop{\rm Hilb}\nolimits}

\newcommand{\Pic}{\mathop{\rm Pic}\nolimits}

\newcommand{\ch}{\mathop{\rm ch}\nolimits}

\newcommand{\td}{\mathop{\rm td}\nolimits}
\newcommand{\Ext}{\mathop{\rm Ext}\nolimits}
\newcommand{\Spec}{\mathop{\rm Spec}\nolimits}

\newcommand{\Coh}{\mathop{\rm Coh}\nolimits}

\newcommand{\cneq}{\mathrel{\raise.095ex\hbox{:}\mkern-4.2mu=}}
\newcommand{\eqcn}{\mathrel{=\mkern-4.5mu\raise.095ex\hbox{:}}}

\newcommand{\Cok}{\mathop{\rm Coker}\nolimits}

\newcommand{\Stab}{\mathop{\rm Stab}\nolimits}

\newcommand{\Sch}{\mathop{\rm Sch}\nolimits}

\newcommand{\modu}{\mathop{\rm mod}\nolimits}

\newcommand{\Imm}{\mathop{\rm Im}\nolimits}
\newcommand{\Ker}{\mathop{\rm Ker}\nolimits}
\newcommand{\Ree}{\mathop{\rm Re}\nolimits}

\newcommand{\GL}{\mathop{\rm GL}\nolimits}

\newcommand{\Auteq}{\mathop{\rm Auteq}\nolimits}

\begin{document}
\maketitle
\begin{abstract}
In this paper, we introduce new enumerative invariants of curves on Calabi-Yau 
3-folds via
 certain stable objects in the derived category of 
coherent sheaves. 
We introduce the notion of limit stability 
on the category of perverse coherent sheaves, 
a subcategory in the derived category, 
and construct the moduli spaces of limit stable objects.   
We then define the counting invariants of limit 
stable objects using Behrend's constructible functions on 
that moduli spaces. 
It will turn out 
that our invariants are generalizations of counting invariants of 
stable pairs introduced by Pandharipande and Thomas.
We will also investigate the wall-crossing phenomena of our 
invariants under change of stability conditions.  
\end{abstract}

\section{Introduction}
The purpose of this paper is to 
introduce new enumerative invariants of curves on Calabi-Yau 
3-folds from certain stable objects in the 
derived category of coherent sheaves. 
The notion of stability conditions on derived categories, 
more generally on triangulated categories,  
is introduced by Bridgeland~\cite{Brs1}, motivated by 
Douglas's work on $\Pi$-stability~\cite{Dou1}, \cite{Dou2}. 
However at this time, there are some issues
in studying Bridgeland's stability conditions on projective 
Calabi-Yau 3-folds. 
 Instead, we consider a generalized 
notion of stability conditions which we call \textit{limit stability}, 
and study their stable objects. 
The limit stability is considered as the ``large volume limit''
in the stringy K$\ddot{\textrm{a}}$hler moduli space.
We construct the moduli spaces of limit stable objects, and
 introduce the enumerative invariants of such objects. 

On the other hand, a kind of enumeration problem of objects in the 
derived category 
is studied by Pandharipande and Thomas~\cite{PT}, \cite{PT2}, \cite{PT3}. 
We will see how our invariants relate to the invariants of 
stable pairs introduced by them~\cite{PT}. 
We will also investigate the wall-crossing phenomena 
of our invariants under change of stability conditions, 
and propose a conjectural wall-crossing formula
 which is related to the rationality 
conjecture proposed in~\cite{PT}.

\subsection{Background}
Let $X$ be a non-singular projective Calabi-Yau 3-fold over $\mathbb{C}$.
The \textit{Gromov-Witten (GW) invariants} of $X$ are counting invariants of curves 
on $X$, integrating over the virtual class of the moduli space of stable 
maps
$\overline{M}(X)$, 
$$\overline{M}(X) =\{ (C, f) \mid f\colon C\to X \mbox{ is a stable 
map from a curve }C\}.$$
Since stable maps have non-trivial automorphisms, $\overline{M}(X)$ 
is in general a Deligne-Mumford stack and the GW invariants 
are rational numbers. 
Another kind of counting invariants of curves on $X$, 
called \textit{Donaldson-Thomas (DT) invariants}, are defined 
as the integration over the virtual class of the moduli 
space of the ideal sheaves, 
$$I(X)= \{ I_C \subset \oO_X \mid C\subset X \mbox{ is a 
subscheme with }\dim C \le 1\}.$$
Since $I(X)$ is nothing but the Hilbert scheme, the resulting 
invariants are integer valued. 
The \textit{GW-DT correspondence}~\cite{MNOP} is a 
conjectural relationship between two generating 
functions involving GW invariants, DT
invariants respectively. More precisely, 
one dimensional subschemes $C\subset X$ 
contain zero dimensional subschemes, hence the DT
theory does not directly count curves. 
Instead by dividing by the generating 
series of counting invariants of zero dimensional 
subschemes, we can define
the reduced DT theory which should correspond to the
GW theory in GW-DT correspondences.  

The notion of stable pairs on $X$ is introduced in~\cite{PT}
in order to give a geometric interpretation 
to the reduced DT theory.  
By definition a stable pair consists of data $(F, s)$, 
\begin{align*}
s\colon \oO_X \lr F, 
\end{align*}
where $F\in \Coh(X)$ is a pure one dimensional sheaf,
and $s$ is a morphism satisfying the condition 
$$\dim \Cok(s)=0.$$
The \textit{Pandharipande-Thomas (PT) invariants} are defined 
by the integration over the virtual class of the 
moduli space of stable pairs, 
$$P(X)=\{(F, s) \mid  s\colon \oO _X \lr F \mbox{ is a stable pair }\}, $$
and the \textit{DT-PT correspondence}~\cite{PT} is a conjectural relationship 
between generating functions of reduced DT theory, 
PT theory respectively. 

On the DT-side, any ideal sheaf $I_C \subset \oO_X$ is 
a Gieseker-stable sheaf, hence DT-invariants count 
stable objects in $\Coh(X)$. On the other hand, 
the space $P(X)$ can be viewed 
as the moduli space of the two term complexes, 
$$I^{\bullet}= \{ \oO_X \stackrel{s}{\to} F \} \in D^b(X), $$
where $D^b(X)$ is the bounded derived category of coherent 
sheaves on $X$. Furthermore the obstruction theory which admits the 
virtual class
on $P(X)$ is obtained from the deformation theory of 
objects in $D^b(X)$, not from that of stable pairs.  
 From this observation, 
we guess that $I^{\bullet}$ might be stable objects 
with respect to a certain stability 
condition on $D^b(X)$, and PT-invariants count stable objects.  

Now we are led to consider stability conditions on $D^b(X)$, 
and enumerative problem of stable objects in $D^b(X)$. 
In the next paragraph,
we discuss stability conditions on derived categories.

\subsection{Stability conditions on triangulated categories}
Let $\dD$ be a triangulated category, e.g. 
$\dD=D^b(X)$ for a smooth projective variety $X$. 
The notion of stability conditions
on $\dD$ is introduced by Bridgeland~\cite{Brs1}. 
Roughly speaking a stability condition on $\dD$
consists of data $\sigma=(Z, \aA)$, 
$$Z\colon K(\dD) \lr \mathbb{C}, \quad 
\aA \subset \dD, $$
where $Z$ is a group homomorphism 
called a stability function, and $\aA$ is the heart of 
a bounded t-structure on $\dD$, which satisfy some 
axiom. When $\dD=D^b(X)$, the set of locally finite numerical stability 
conditions $\Stab(X)$ is shown to have the complex 
structure by Bridgeland~\cite{Brs1}, 
and the quotient space
$$
\Auteq (\dD) \backslash \Stab(X)/ \mathbb{C}$$
is a mathematical candidate of the 
stringy K$\ddot{\textrm{a}}$hler moduli space. 
The space $\Stab(X)$ have been studied
in several examples. For instance see~\cite{Brs2}, \cite{Brs3}, 
\cite{Brs4}, \cite{IUU}, \cite{Mac}, \cite{Tho},  
\cite{Tst}, \cite{Tst2}.

Although the notion of stability conditions on triangulated
categories has 
drawn much interest recently, 
we are not able to study the most important case, 
$\dD=D^b(X)$ for 
a projective Calabi-Yau 3-fold $X$
at this time. 
In this case, there are some technical difficulties
to construct examples of stability conditions, so we do not 
know whether $\Stab(X)$ is non-empty or not. 
From the ideas in physical articles~\cite{Dou1}, \cite{Dou2},
there should exist stability conditions
corresponding to the neighborhood of the large volume limits, 
whose stability functions are 
given by, 
\begin{align}\label{charge}
Z_{\sigma}(E)=-\int e^{-(B+i\omega)}\ch(E)\sqrt{\td_X},\end{align}
where $\sigma=B+i\omega \in H^2(X, \mathbb{C})$ with $\omega$ 
an ample class.
Such stability conditions should be parameterized by 
elements of the complexified ample cone, 
$$\sigma \in A(X)_{\mathbb{C}}=\{ B+i\omega \in H^2(X, \mathbb{C}) \mid 
\omega \mbox{ is an ample class} \}.$$
Instead of working with Bridgeland's stability 
conditions, we introduce and study a generalized 
notion of stability conditions which we call \textit{limit 
stability}. The corresponding heart of a
t-structure is the category of perverse coherent sheaves, 
$$\aA^p \subset D^b(X), $$
in the sense of Bezrukavnikov~\cite{Bez} and Kashiwara~\cite{Kashi}.
We will see that 
for $\sigma \in A(X)_{\mathbb{C}}$, 
the stability function (\ref{charge}) 
together with taking $\omega \to \infty$
 determines the set of (semi)stable objects
in $\aA^p$, 
which we call $\sigma$-\textit{limit (semi)stable objects}. 
The notation ``limit'' is used to emphasize that
 our stability conditions should correspond to the
 limit point $\omega = \infty$.  
Some fundamental properties of limit stability
(e.g. existence of Harder-Narasimhan filtrations, Jordan-H\"older filtrations,)
will be studied in Section~\ref{section:limit}.

\subsection{Main results}
We shall study the enumerative problem of 
$\sigma$-limit stable objects $E\in \aA^p$. 
Let us take $\beta \in H^4(X, \mathbb{Q})$ and 
$n\in \mathbb{Q}$. 
We first show the existence of the moduli space of 
limit stable objects. 
The following theorem will be shown in Section~\ref{section:moduli}. 
\begin{thm}\label{main1}
There is a separated algebraic space of finite type
$\lL_n^{\sigma}(X, \beta)$, 
which parameterizes $\sigma$-limit stable objects 
$E\in \aA^p$, satisfying $\det E=\oO_X$ and the  
following numerical condition, 
$$(\ch_0(E), \ch_1(E), \ch_2(E), \ch_3(E))=(-1, 0, \beta, n)
\in H^{\ast}(X, \mathbb{Q}).$$
\end{thm}
It will turn out that the moduli space $\lL_n^{\sigma}(X, \beta)$
could be non-empty only if $\beta$ is the Poincar\'e dual 
of the homology class of an effective one
 cycle on $X$, and $n\in \mathbb{Z}$. 
(cf.~Remark~\ref{turnout}.)
By Theorem~\ref{main1}
and using Behrend's constructible function~\cite{Beh}, 
$\nu_{L}\colon \lL_n^{\sigma}(X, \beta) \to \mathbb{Z}$, we 
are able to define the counting 
invariant of limit stable objects,
$$L_{n, \beta}(\sigma)\cneq \sum_{n\in \mathbb{Z}}ne(\nu_{L}^{-1}(n))
 \in \mathbb{Z}.$$
We next show the relationship between the integers
$L_{n, \beta}(\sigma)$ and $P_{n, \beta}$, where 
$P_{n, \beta}$ is the PT-invariant 
counting stable pairs 
$(F, s)$ with 
$$\ch_2(F)=\beta, \quad \ch_3(F)=n.$$
See Definition~\ref{PTinv} for the detail. 
We show the following theorem in Section~\ref{section:count}. 
\begin{thm}\label{main2}
Let $\sigma=k\omega +i\omega$ for $k\in \mathbb{R}$.
We have, 
$$L_{n, \beta}(\sigma)=P_{n, \beta}, \quad (k\ll 0), \qquad 
L_{n, \beta}(\sigma)=P_{-n, \beta}, \quad (k\gg 0).
$$
\end{thm}
It seems that Theorem~\ref{main2} is related to the rationality conjecture 
of the generating function of the PT-invariants, 
$$Z_{\beta}^{\rm{PT}}(q)=\sum_{n\in \mathbb{Z}}P_{n, \beta}q^n
\in \mathbb{Q}\db[q\db].$$
It is proposed by Pandharipande and Thomas
in~\cite[Conjecture 3.2]{PT} and they conjecture that 
$Z_{\beta}^{\rm{PT}}(q)$ is a rational function of $q$, invariant 
under $q \mapsto 1/q$. 
This conjecture is solved when $\beta$ is an irreducible curve class
in~\cite{PT3}
by comparing $P_{n, \beta}$ and $P_{-n, \beta}$. 
In the following, we propose a conjectural wall-crossing 
formula of our invariants 
$L_{n, \beta}(\sigma)$, which combined with Theorem~\ref{main2}
 provides a relationship between 
$P_{n, \beta}$ and $P_{-n, \beta}$ in a general situation.
For $\mu\in \mathbb{Q}$, let $k_0=-\mu/2$
and $k_{-}<k_0$, $k_{+}>k_0$ are sufficiently close to $k_0$. 
We set $\sigma_{\ast}=k_{\ast}\omega +i\omega$ for 
$\ast=0, \pm$. 
\begin{conj}\label{conj:intro}
There is a virtual counting of one dimensional 
$\omega$-Gieseker semistable sheaves $F$ 
with $(\ch_2(F), \ch_3(F))=(\beta', n')$, denoted by 
$N_{n', \beta'}\in \mathbb{Q}$, such that 
\begin{align}\label{expect2}L_{n, \beta}(\sigma_{-})-L_{n, \beta}(\sigma_{+})=
\sum (-1)^{n'-1}n'N_{n', \beta'}L_{n'', \beta''}(\sigma_0).\end{align}
Here in the above sum, $(\beta', n')$, $(\beta'', n'')$ must
satisfy 
$\beta'+\beta''=\beta$, 
 $n'+n''=n$ and $n'/\omega\beta'=\mu$. 
\end{conj}
See Paragraph~\ref{sub:wall} for the explanation of the above conjecture. 
In Section~\ref{section:Ex}, we investigate the wall-crossing 
phenomena of limit stable objects and study Conjecture~\ref{conj:intro}
in some examples.

\subsection{Acknowledgement}
The author thanks J.~Li and R.~Pandharipande 
for useful conversations at the conference 
``Integrable Systems and Mirror symmetry"
on January 2008 at Kyoto.
He also thanks R.~Thomas for nice comments, 
and the referee for pointing out an easier 
proof of Lemma~\ref{-infty}. 
While the author was preparing the manuscript, 
the notion of stability conditions on perverse 
coherent sheaves was introduced by
A.~Bayer~\cite{Bay} in more general situation independently. 
He thanks E.~Macr\`i for the information of 
Bayer's work.  
He is partially supported by Japan Society for the 
Promotion of Sciences Research Fellowship for Young Scientists, 
No. 198007.

\subsection{Notation and convention}
We work over varieties over $\mathbb{C}$. 
For a variety $X$, 
we denote by $D^b(X)$, $K(X)$ the bounded derived category of 
coherent sheaves on $X$, the Grothendieck group of coherent 
sheaves respectively. 
For a triangulated category $\dD$ and a set of
subobjects $S\subset \dD$, we denote by 
$\langle S \rangle \subset \dD$
 the smallest extension closed subcategory 
 which contains $S$. 
 If $S$ is a set of subobjects in an abelian category $\aA$, 
 we also use the same notation $\langle S \rangle \subset \aA$.

\section{Limit stability}\label{section:limit}
\subsection{Bridgeland's stability conditions}
Here we briefly review the definition of Bridgeland's stability 
conditions~\cite{Brs1}.  
Let us begin with the stability conditions on abelian categories. 
\begin{defi}\emph{\bf{\cite{Brs1}}}\emph{
Let $\aA$ be an abelian category. A \textit{stability function} on 
$\aA$ is a group homomorphism, 
$$Z\colon K(\aA) \lr \mathbb{C}, $$
such that for any non-zero $E\in \aA$, we have 
$$Z(E) \in 
\mathbb{H} \cneq \{ r\exp(i\pi \phi) \mid r>0, 0<\phi \le 1 \}.$$}
\end{defi}
Given a non-zero object $E\in \aA$ and a 
 stability function $Z\colon K(\aA) \to \mathbb{C}$, 
we can uniquely determine the \textit{phase} of $E$ by 
$$\phi(E) = \frac{1}{\pi}\Imm \log Z(E) \in (0, 1].$$
We say $E\in \aA$ is $Z$-\textit{semistable} if for any non-zero 
subobject $F\subset E$ in $\aA$, we have 
$$\phi(F) \le \phi(E).$$
\begin{defi}\emph{\bf{\cite{Brs1}}}\emph{
A stability function $Z\colon K(\aA) \to \mathbb{C}$ is called 
a \textit{stability condition} on $\aA$ if for any $E\in \aA$, there 
exists a filtration 
$$0=E_0 \subset E_1 \subset \cdots \subset E_n=E, $$
such that each $F_i=E_i/E_{i-1}$ is $Z$-semistable with 
$$\phi(F_1)>\phi(F_2) > \cdots > \phi(F_n).$$
The above filtration is called a \textit{Harder-Narasimhan filtration}.}
\end{defi}
It is easy to construct examples of stability conditions
if $\aA$ has finite number of simple objects 
$S_1, \cdots, S_N \in \aA$
such that 
$$\aA=\langle S_1, \cdots, S_N \rangle.$$
e.g. $\aA=\modu A$ for a finite dimensional $k$-algebra $A$. 
In this case $K(\aA)$ is generated by $[S_i]\in K(\aA)$, and 
$Z\colon K(\aA) \to \mathbb{C}$ is a stability condition if 
and only if $Z(S_i) \in \mathbb{H}$ for all $i$.

In general, a sufficient condition for 
a stability function 
to be a stability condition is provided in~\cite[Proposition 2.4]{Brs1}. 
\begin{prop}\emph{\bf{\cite[Proposition 2.4]{Brs1}}}\label{suff}
Let $Z\colon K(\aA)\to \mathbb{C}$ be a stability function. 
Assume that 
\begin{itemize}
\item there is no infinite sequence of inclusions in $\aA$, 
\begin{align}\label{inf1}
 \cdots \hookrightarrow E_{n} \hookrightarrow 
\cdots \hookrightarrow E_1 \hookrightarrow E_0, \end{align}
with $\phi(E_{i+1})>\phi(E_i)$ for all $i$.
\item there is no infinite sequence of 
surjections in $\aA$, 
\begin{align}\label{inf2}
E_0 \twoheadrightarrow E_1 
\twoheadrightarrow \cdots \twoheadrightarrow 
E_n \twoheadrightarrow \cdots, \end{align}
with $\phi(E_i)>\phi(E_{i+1})$ for all $i$. 
\end{itemize}
Then $Z$ is a stability condition. 
\end{prop}

Next let $\dD$ be a triangulated category, e.g. 
$\dD=D^b(X)$ for a variety $X$. 
The following is the definition of Bridgeland's stability 
conditions. 
\begin{defi}\emph{\bf{\cite{Brs1}}}\label{def:stab}\emph{
A stability condition on $\dD$ consists of data $(Z, \aA)$, 
where $\aA \subset \dD$ is the heart of a bounded t-structure on 
$\dD$, and $Z$ is a stability condition on $\aA$. 
}\end{defi}

\begin{rmk}\emph{
A stability condition on $\dD$
in~\cite{Brs1} is originally given by 
data $(Z, \pP)$, where 
$Z\colon K(\dD) \to \mathbb{C}$ is a group homomorphism, 
and $\pP(\phi)\subset \dD$ for $\phi \in \mathbb{R}$ are full subcategories, 
satisfying some axioms. However as shown in~\cite[Proposition 4.2]{Brs1},  
this is equivalent to giving data $(Z, \aA)$ as in Definition~\ref{def:stab}.}
\end{rmk}
\begin{rmk}\emph{
Let $Z \colon K(\aA) \to \mathbb{C}$ be a stability 
condition on an abelian category $\aA$. 
Then the pair $(Z, \aA)$ is a stability condition on 
the triangulated category $D^b(\aA)$. }
\end{rmk}
Let $\dD=D^b(X)$ for a smooth projective variety. 
A stability condition $(Z, \aA)$ on $\dD$ is called 
\textit{numerical}
if $Z\colon K(X) \to \mathbb{C}$ factors through the Chern character 
map, 
$$\xymatrix{
K(X) \ar[r]^{Z} \ar[d]_{\ch} & \mathbb{C}. \\
H^{\ast}(X, \mathbb{Q}) \ar[ru] }$$
The set of numerical
 stability conditions 
on $\dD=D^b(X)$ which satisfy the \textit{local finiteness}
(cf.~\cite[Definition 5.7]{Brs1}) is denoted by $\Stab(X)$. 
In~\cite[Theorem 1.2]{Brs1},
Bridgeland shows that $\Stab(X)$ has a
structure of a complex manifold.

The space $\Stab(X)$ is studied when $\dim X=1, 2$ 
in the articles~\cite{Brs1}, \cite{Brs2}. Unfortunately, 
we do not know how to construct examples of stability conditions 
for higher dimensional varieties. It seems that the following 
lemma is well-known, but we put it to emphasize that 
the construction problem is non-trivial. 
\begin{lem}
Let $X$ be a smooth projective variety
with $d=\dim X \ge 2$.  
Then there is no numerical stability condition 
$(Z, \aA)$ on $D^b(X)$ with $\aA=\Coh(X)$. 
\end{lem}
\begin{proof}
It is enough to show that there is no stability 
function $Z\colon K(X) \to \mathbb{C}$ on $\Coh(X)$ of the 
following form, 
$$Z(E)=\sum_{j=0}^d (u_j+iv_j)\ch_j(E),$$
for $u_j+iv_j \in H^{2d-2j}(X, \mathbb{C})$. 
Suppose that such a stability function $Z$ exists. 
Since $d\ge 2$, there is a smooth subvariety 
$S\stackrel{i}{\hookrightarrow} X$ with $\dim S=2$. Then the composition 
$$K(S) \stackrel{i_{\ast}}{\lr} K(X) \stackrel{Z}{\lr} \mathbb{C}, $$
is a stability function on $\Coh(S)$, hence we may 
assume $d=2$. Let $C\subset X$ be a smooth curve
and take a divisor $D$ on $C$.  
Since $\Imm Z(E) \ge 0$ for any $E\in \Coh(X)$, we have 
$$\Imm Z(\oO_C(D))=v_2(\deg D +\ch_2(\oO_C))+v_1 \cdot [C] \ge 0.$$
Since we can take $D$ with an arbitrary degree, 
we must have $v_2=0$. Similarly 
we have 
$$\Imm Z(\oO_X(mC))=mv_1 \cdot [C] +v_0 \ge 0,$$
for any $m\in \mathbb{Z}$, hence $v_1 \cdot [C]=0$. 
Therefore $\Imm Z(\oO_C(D))=0$, and this implies
$$\Ree Z(\oO_C(D))=
u_2(\deg D +\ch_2(\oO_C))+u_1 \cdot [C] \le 0,$$
since $Z(\oO_C(D)) \in \mathbb{H}$. 
Then the same argument shows that $u_2=0$, and this implies
$$Z(\oO_x)=u_2+iv_2 =0,$$
for any closed point $x\in X$. This contradicts that 
$Z(\oO_x) \in \mathbb{H}$. 
\end{proof}

\begin{rmk}\label{rmk:find}\emph{
In the case of $\dim X=2$, the examples of 
stability conditions $(Z, \aA)$ are constructed in~\cite{Brs2}, \cite{AB}
by setting $\aA$ to be the tilting of $\Coh(X)$ with respect
to certain torsion pairs.
When $X$ is a K3 surface, the stability function $Z$ is given by 
$$Z_{(B, \omega)}(E)=-\int e^{-(B+i\omega)}\ch(E) \sqrt{\td_X}, $$
for $B+i\omega \in H^2(X, \mathbb{C})$ with $\omega$ an ample
class. When $\dim X\ge 3$ we expect that for
 $\omega \gg 0$, there are hearts of 
bounded t-structures $\aA_{(B, \omega)}$ such that the pairs 
$(Z_{(B, \omega)}, \aA_{(B, \omega)})$ determine stability conditions, 
giving the neighborhood of the large volume limits.  
However at this time, we are not able to find such 
$\aA_{(B, \omega)}$. 
}
\end{rmk}

\subsection{Perverse coherent sheaves on Calabi-Yau 3-folds}
From this paragraph, we focus on the case that 
$X$ is a Calabi-Yau 3-fold, i.e. $X$ is a smooth projective 
3-fold with a trivial canonical class. 
Here we study the heart of a bounded t-structure $\aA^p$, 
constructed as one of the perverse t-structures introduced by 
Bezrukavnikov~\cite{Bez} and Kashiwara~\cite{Kashi}.
In the notation of Remark~\ref{rmk:find}, the desired category 
$\aA_{(B, \omega)}$ should be constructed as an approximation of 
our category $\aA^p$, so we hope that
studying $\aA^p$ in detail will solve the construction problem in a 
future. Let us recall the notion of 
torsion pairs and their tilting for the construction of $\aA^p$. 

\begin{defi}\emph{Let $\aA$ be an abelian category. 
A \textit{torsion pair} on $\aA$ is a pair of full subcategories 
$(\tT, \fF)$ such that
\begin{itemize}
\item For $T\in \tT$ and $F\in \fF$, we have $\Hom(T, F)=0$. 
\item For any $E\in \aA$, there is an exact sequence 
$0 \to T \to E \to F \to 0$ in $\aA$ such that 
$T\in \tT$, $F\in \fF$. 
\end{itemize}
}
\end{defi}
Given a torsion pair $(\tT, \fF)$ on $\aA$, 
the following subcategory of $D^b(\aA)$,  
\begin{align*}
\aA^{\dag} 
&=\langle \fF[1], \tT \rangle \\
&=\{ E\in D^b(\aA) \mid \hH^{-1}(E)\in \fF, \hH^0(E) \in \tT, 
\hH^i(E)=0 \mbox{ for }i\neq -1, 0\}, 
\end{align*}
is known to be the heart of a bounded t-structure on $D^b(\aA)$, 
and it is called a \textit{tilting}
with respect to the torsion pair $(\tT, \fF)$. (cf.~\cite{HRS}.)
For a Calabi-Yau 3-fold $X$, we have the 
following torsion pair.  

\begin{lem}\emph{
The pair $(\Coh_{\le 1}(X), \Coh_{\ge 2}(X))$,
\begin{align*}
\Coh_{\le 1}(X) & \cneq \{ E\in \Coh(X) \mid \dim \Supp(E) \le 1\}, \\
\Coh_{\ge 2}(X) & \cneq \{ E\in \Coh(X) \mid \Hom(F, E)=0 
\mbox{ for any }
F\in \Coh_{\le 1}(X)\},
\end{align*}
is a torsion pair of $\Coh(X)$.}
\end{lem}
\begin{proof}
For an object $E\in \Coh(X)$
there is an exact sequence, 
$$0 \lr T \lr E \lr F\lr 0, $$
such that 
$T\in \Coh_{\le 1}(X)$ and $\dim \Supp(F)\ge 2$. Since $\Coh(X)$ is 
a noetherian abelian category, we can take $T$ to be maximum, i.e. 
there is no $T' \in \Coh_{\le 1}(X)$
with $T\subsetneq T' \subset E$. 
Then it is easy to see that $F\in \Coh_{\ge 2}(X)$. 
\end{proof}
Our abelian category $\aA^p$ is constructed as a tilting. 
\begin{defi}\emph{
We define the \textit{heart of a perverse t-structure}
$\aA^p \subset D^b(X)$ to be the tilting with respect to the torsion pair
$(\Coh_{\le 1}(X), \Coh_{\ge 2}(X))$, i.e. 
$$\aA^p =\langle \Coh_{\ge 2}(X)[1], \Coh_{\le 1}(X) \rangle.$$
}
\end{defi}
\begin{rmk}\emph{
In general, a perverse t-structure introduced in~\cite{Bez}, \cite{Kashi}
is
determined by choosing a perversity function, which is 
a map 
$p\colon X^{\rm top} \to \mathbb{Z}$ satisfying 
a certain condition. One can 
easily check that our category $\aA^p$ corresponds to 
 the following perversity function, 
\begin{align*}
p(x)= \left\{ \begin{array}{cc}
-1 & \quad \dim \oO _{X, x}\le 1, \\
0 & \quad \dim \oO_{X, x} \ge 2. \end{array}
 \right. 
\end{align*}}
\end{rmk}
\begin{rmk}\label{rmk:art}\emph{
The subcategory $\Coh_{\le 1}(X) \subset \aA^p$ is
easily seen to be closed under quotients and subobjects, 
hence it is an abelian subcategory.
Since $\Coh_{\le 1}(X)$ is not artinian, the abelian category 
$\aA^p$ is also not artinian. 
 }
\end{rmk}
\begin{rmk}\label{rmk:noe}\emph{
The abelian category $\aA^p$ is also not noetherian. 
In fact let us take a divisor $H\subset X$ and a curve 
$C\subset H$. Then there exists an infinite chain of surjections 
in $\aA^p$, 
$$\oO_H[1] \twoheadrightarrow \oO_H(C)[1] \twoheadrightarrow \oO_H(2C)[1] \twoheadrightarrow \cdots.$$
}
\end{rmk}

\subsection{Torsion pair on $\aA^p$ and the dualizing functor}
As we have seen in Remark~\ref{rmk:noe}, 
the abelian category $\aA^p$
is worse than $\Coh(X)$, and this fact sometimes causes
difficulty to handle $\aA^p$.
In this paragraph, we introduce a certain torsion pair on $\aA^p$
which makes $\aA^p$ much more amenable.  
Let us set $\aA^p_1$, $\aA^p_{1/2}$ to be the subcategories of $\aA^p$, 
\begin{align*}
\aA^p_1 & \cneq \langle F[1], \oO_x \mid F \mbox{ is a pure two dimensional 
sheaf and }x\in X\rangle, \\
\aA^p_{1/2} & \cneq \{ E\in \aA^p \mid \Hom(F, E)=0
\mbox{ for any }F\in \aA_{1}^p \}.
\end{align*}
The meaning of the subscript will be clear in Lemma~\ref{infty}.
\begin{rmk}\label{rmk:ob}\emph{
For $E\in \aA^p$, it is obvious that $E\in \aA_1^p$ if and only 
if $\hH^0(E)$ is zero dimensional and $\hH^{-1}(E)$ is 
a torsion sheaf. Also $E\in \aA_{1/2}^p$ if and only if 
$\hH^{-1}(E)$ is torsion free and $\Hom(\oO_x, E)=0$ for 
any $x\in X$. In particular, we have 
$$\aA_{1/2}^p \cap \Coh_{\le 1}(X) =\{\mbox{\rm{pure one dimensional 
sheaves}}\}.$$} 
\end{rmk}
We show the following lemma. 
\begin{lem}\label{tor}
The pair $(\aA^p_{1}, \aA^p_{1/2})$ is a torsion pair of $\aA^p$. 
\end{lem}
\begin{proof}
It is enough to show that for any 
$E\in \aA^p$, 
there is an exact sequence in $\aA^p$
$$0 \lr E_{1} \lr E \lr E_{1/2} \lr 0, $$
with $E_i \in \aA_{i}^p$ for $i=1, 1/2$. 
Let $F\subset \hH^{-1}(E)$ be the 
maximum torsion subsheaf. We have the exact sequence in $\aA^p$, 
$$0 \lr F[1] \lr E \lr E' \lr 0.$$
Note that $F[1] \in \aA_{1}^p$
and $\hH^{-1}(E')=\hH^{-1}(E)/F$ is torsion free. 
Hence for any pure two dimensional sheaf 
$F'$, we have
$$\Hom(F'[1], E')=\Hom(F', \hH^{-1}(E'))=0.$$
Therefore
if $E'$ is not contained in $\aA_{1/2}^p$, there 
is a zero dimensional sheaf $\uU$ such that 
$\Hom(\uU, E')\neq 0$. By Remark~\ref{rmk:art}, 
this means that there is a subobject $\uU' \subset E'$ 
in $\aA^p$ such that $\uU'$ is a zero dimensional sheaf. 
Moreover we can take 
$\uU'\subset E$ to be maximum, 
i.e. there is no zero dimensional sheaf $\uU''$ with
$\uU' \subsetneq \uU'' \subset E'$ in $\aA^p$. 
To show this, it is enough to check 
that any sequence of 
 subobjects, 
\begin{align}\label{chain}
\uU_1 \subset \uU_2 \subset \cdots \subset \uU_n \subset \cdots \subset E', 
\end{align}
where $\uU_i$ are zero dimensional sheaves, terminates.  
Let $G_i=E'/\uU_i \in \aA^p$. We have the exact 
sequence in $\aA^p$, 
$$0 \lr \uU_i/\uU_{i+1} \lr G_i \lr G_{i+1} \lr 0.$$
Taking cohomology and noting that $\hH^{-1}(E')$ is torsion free, we see that 
\begin{align}\label{inf}
\hH^{-1}(E') \subset \hH^{-1}(G_1) \subset \hH^{-1}(G_2) 
\subset \cdots \subset \hH^{-1}(G_n) \subset \cdots \subset \hH^{-1}(E')^{\vee 
\vee},\end{align}
in $\Coh(X)$. 
The sequence (\ref{inf}) must terminate, 
say $\hH^{-1}(G_j)=\hH^{-1}(G_{j+1})=\cdots$.
Replacing $E'$ by $G_j$, we may assume that
$\hH^{-1}(E') = \hH^{-1}(G_{i})$ for any $i$. 
Then each $\uU_i$ are subsheaves of $\hH^0(E')$, thus (\ref{chain}) 
must terminate. Therefore there is a maximum zero dimensional sheaf 
$\uU' \subset E'$.  

Now let $E''=E'/\uU'$, and consider the exact sequences in $\aA^p$, 
\begin{align*}0 \lr F' \lr E \lr E'' \lr 0, \\
0 \lr F[1] \lr F' \lr \uU' \lr 0.
\end{align*}
Here $E \twoheadrightarrow
 E''$ is obtained as the composition of the quotients in $\aA^p$, 
$E \twoheadrightarrow E' \twoheadrightarrow E''$. 
The bottom sequence shows $F' \in \aA_1^p$. 
By the construction, we also have 
$E'' \in \aA^p_{1/2}$.
\end{proof}
Let $\mathbb{D}\colon D^b(X) \to D^b(X)^{\mathrm{op}}$ be 
the dualizing functor, 
$$\mathbb{D}(E)=\dR \hH om(E, \oO_X[2]).$$
In the following lemma, we see the compatibility 
of the torsion pair $(\aA_{1}^p, \aA_{1/2}^p)$ with
the dualizing functor $\mathbb{D}$. 
\begin{lem}\label{dual}
We have 
\begin{align*}
 E\in \aA^p_1 & \Rightarrow \mathbb{D}(E) \in \aA^p_1[-1], \\
 E\in \aA^p_{1/2} & \Rightarrow \mathbb{D}(E) \in \aA^p_{1/2}.
\end{align*}
\end{lem}
\begin{proof}
First we show that $\mathbb{D}(E) \in \aA^p_1[-1]$ 
for $E\in \aA^p_1$. It is enough to check this 
for $E=G[1]$ and $E=\oO_x$, where 
$G$ is a pure two dimensional sheaf and
$x\in X$ is a closed point. 
Since $G$ is pure, we have 
$$\eE xt^i_X(G, \oO_X)=0, \quad \mbox{ for }i\neq 1,$$ and 
$\eE xt^1_X(G, \oO_X)$ is a pure two dimensional sheaf. 
(cf.~\cite[Section 1.1]{Hu}.)
Therefore $\mathbb{D}(G[1])\in \aA^p_1[-1]$. 
Also we have $\mathbb{D}(\oO_x)=\oO_x[-1] \in \aA^p_1[-1]$. 

Next let us take $E\in \aA^p_{1/2}$ and check 
$\mathbb{D}(E)\in \aA^p_{1/2}$.  
Since 
$\hH^0(E)$ is a torsion sheaf and 
$E$ is concentrated on $[-1,0]$, 
we can easily see $\hH^i(\mathbb{D}(E))=0$ for $i\le -2$.
Suppose that $\hH^k(\mathbb{D}(E)) \neq 0$
and $\hH^i(\mathbb{D}(E))=0$ for any $i<k$. 
Then there is a closed point $x\in X$ such that 
$$0 \neq \Hom(\mathbb{D}(E), \oO_x[-k]) =
\Hom(\oO_x[k-1], E).$$
Therefore we have $k\le 0$, and $\mathbb{D}(E)$ is concentrated on
$[-1,0]$. Let us take $F\in \Coh_{\le 1}(X)$. Since we have 
$$\hH^i(\mathbb{D}(F[1])) =0 \quad \mbox{ for }
i \le 0, $$
it follows that 
\begin{align*}
\Hom(F, \hH^{-1}(\mathbb{D}(E))) &=\Hom(F[1], \mathbb{D}(E)) \\
&= \Hom(E, \mathbb{D}(F[1])) \\
&=0.\end{align*}
Hence $\hH^{-1}(\mathbb{D}(E)) \in \Coh_{\ge 2}(X)$. 
Let us take a codimension one point
$p\in X$. Since $\hH^{-1}(E)$ is 
torsion free, we 
have $E_p \cong \oO_{X,p}^{\oplus r}[1]$ for some $r$. 
Therefore $\mathbb{D}(E)_p \cong \oO_{X,p}^{\oplus r}[1]$, and 
this implies $\hH^0(\mathbb{D}(E)) \in \Coh_{\le 1}(X)$, i.e. 
$\mathbb{D}(E) \in \aA^p$. Moreover for any object $E' \in \aA^p_{1}$,
we have
$$ \Hom(E', E) \cong \Hom(\mathbb{D}(E), \mathbb{D}(E'))=0,$$
since $\mathbb{D}(E') \in \aA^p_{1}[-1]$. Therefore we can conclude
$\mathbb{D}(E)\in \aA^p_{1/2}$. 
\end{proof}
\begin{rmk}\emph{
According to~\cite{Kashi}, the abelian category 
$\mathbb{D}(\aA^p)$ corresponds to the 
heart of a perverse t-structure $\aA^{p^{\ast}}$ (up to shift)
with the dual 
perversity function $p^{\ast}\colon X^{\rm top} \to \mathbb{Z}$. 
Lemma~\ref{dual} implies that $\aA^{p^{\ast}}$ is obtained as 
a tilting with respect to the torsion pair $(\aA_{1}^p, \aA_{1/2}^p)$.} 
\end{rmk}

Let us take $E, F\in \aA^p_i$ and a morphism  
$f\colon E\to F$. The morphism $f$ is called a \textit{strict monomorphism}
if $f$ is injective in $\aA^p$ and $\Cok(f) \in \aA^p_i$. 
Similarly $f$ is called a \textit{strict epimorphism} if 
$f$ is surjective in $\aA^p$ and $\Ker(f) \in \aA^p_i$. 
Although the category $\aA^p$ is not artinian nor noetherian, 
each subcategories $\aA^p_i$ have such properties. 

\begin{lem}\label{finlen}
For $i=1, 1/2$, the category $\aA^p_i$ is of finite length
with respect to strict monomorphisms, and strict epimorphisms, i.e. 
any infinite chains of strict monomorphisms, strict 
epimorphisms in $\aA^p_i$, 
\begin{align}\label{mono}
&\cdots \hookrightarrow E_n  \hookrightarrow 
\cdots \hookrightarrow E_1 \hookrightarrow E_0, \\
\label{epi}
& E_0 \twoheadrightarrow E_1 \twoheadrightarrow \cdots \twoheadrightarrow E_n 
\twoheadrightarrow \cdots.
\end{align}
must terminate.  
\end{lem}
\begin{proof}
By applying the dualizing functor $\mathbb{D}$ and using 
Lemma~\ref{dual}, it is enough to 
show that a chain (\ref{mono}) terminates. 
Let us take an infinite chain 
(\ref{mono}) in $\aA^p_1$ with each $E_i \in \aA^p_1$. 
Let $\omega$ be an ample divisor on $X$. 
Since $-\ch_1(E) \cdot \omega ^2 \ge 0$ for $E\in \aA^p_1$, 
we have 
$$-\ch_1(E_i)\cdot \omega ^2 \ge -\ch_1(E_{i+1})\cdot \omega ^2 \ge 0.$$ 
Hence we may assume that $\ch_1(E_i)\cdot \omega ^2 = 
\ch_1(E_{i+1})\cdot \omega ^2$ for any $i$,  and this implies that 
the induced morphism 
$$\hH^{-1}(E_i) \lr \hH^{-1}(E_{i+1}),$$
 is an isomorphism in 
codimension one. Let us take the exact sequence in $\aA^p$, 
\begin{align}\label{G}
0 \lr E_i \lr E_{i+1} \lr G_i \lr 0.\end{align}
Then $\hH^{-1}(G_i)=0$ since otherwise $\hH^{-1}(G_i)$ 
is one or zero dimensional, and contradicts that
$\hH^{-1}(G_i) \in \Coh_{\ge 2}(X)$.
Taking the cohomology of (\ref{G}), 
we have the chain of inclusions of sheaves, 
\begin{align}
\label{mono2}
 \cdots \subset \hH^0(E_n) \subset \cdots \subset \hH^0(E_1) \subset
\hH^0(E_0).\end{align}
The sequence (\ref{mono2}) must terminate since each 
$\hH^0(E_j)$ is a zero dimensional sheaf by the definition of $\aA^p_1$. 
Hence the chain (\ref{mono}) also terminates. 

Similarly let us take a chain (\ref{mono})
with each $E_i \in \aA^p_{1/2}$. Then we have 
$$-\ch_0(E_i) \ge -\ch_0(E_{i+1}) \ge 0,$$
 hence we may assume $-\ch_0(E_i)=-\ch_0(E_{i+1})$ for any $i$. 
Let us consider the exact sequence as in (\ref{G}). Again 
$\hH^{-1}(G_i)=0$ since otherwise it 
is a two dimensional sheaf, which contradicts that $G_i \in \aA^p_{1/2}$. 
Taking the cohomology of (\ref{G}), 
we obtain the sequence (\ref{mono2}). In this 
case, we have 
$$\ch_2(\hH^0(E_i)) \cdot \omega \ge \ch_2(\hH^0(E_{i+1})) \cdot \omega
\ge 0,$$
hence we may assume $\ch_2(\hH^0(E_i))\cdot \omega =
\ch_2(\hH^0(E_{i+1}))\cdot \omega$. Then $G_i=\hH^0(G_i)$ is zero 
dimensional, thus $G_i=0$ by the definition of $\aA^p_{1/2}$. 
Therefore (\ref{mono}) must terminate. 
\end{proof}

\subsection{Limit stability on $\aA^p$}
Here we introduce the notion of limit stability
on $\aA^p$. Let
$A(X)_{\mathbb{C}}$ be the complexified ample cone, 
$$A(X)_{\mathbb{C}}\cneq \{
B+i\omega \in H^2(X, \mathbb{C}) \mid \omega \mbox{ is an ample 
class }\}.$$
For $\sigma=B+i\omega \in A(X)_{\mathbb{C}}$, we 
consider the group homomorphism $Z_{\sigma}\colon K(X) \to 
\mathbb{C}$, 
\begin{align}\label{charge}
Z_{\sigma}(E)=-\int e^{-(B+i\omega)}\ch(E) \sqrt{\td_X}.\end{align}
The above function does not give a stability function on 
$\aA^p$. However if we replace $\sigma$ by 
$$\sigma_{m}=B+mi\omega, \quad \mbox{ for } m\gg 0, $$
then we can define the well-defined argument 
 of $Z_{\sigma_m}(E)$ for any non-zero $E\in \aA^p$, 
 which defines the set of (semi)-stable objects in $\aA^p$.
 To see this in more detail, let us introduce the 
 (twisted) Mukai vector, 
 $$v^B \colon K(X) \ni E \longmapsto
 e^{-B}\ch(E) \sqrt{\td_X} \in H^{\ast}(X, \mathbb{R}).$$
 Let $v_i^B(E) \in H^{2i}(X, \mathbb{R})$
 be the $H^{2i}$-component of $v^B(E)$. Then one
 can expand (\ref{charge}) and give the following formula, 
 \begin{align}\label{formula}
Z_{\sigma_m}(E)&=-\int e^{-mi\omega}v^B(E) \\
\label{formula1}&=\left( -v^B_{3}(E)+\frac{1}{2}m^2 \omega ^2 v^B_1(E) \right) 
+ \left( m\omega v_2^{B}(E) -\frac{1}{6}m^3 \omega^3 v^B_0(E)\right)i.
\end{align}
We have the following lemma. 
\begin{lem}\label{lem:asy}
For a non-zero object $E\in \aA^p$, we have 
\begin{align}\label{asy}Z_{\sigma_m}(E) \in
\left\{ r\exp(i\pi \phi) : r>0, \frac{1}{4}
< \phi < \frac{5}{4}\right\},\end{align}
for $m\gg 0$. 
\end{lem}
\begin{proof}
Let us take $E\in \Coh(X)$ with $\dim \Supp(E)=3-i$ for 
$0 \le i\le 3$. It is easy to see that 
$v_j^B(E)=0$ for $j<i$ and 
$$v_i^B(E)\cdot \omega^{3-i}=\ch_i(E) \cdot \omega ^{3-i} >0.$$
Therefore by the formula (\ref{formula1}), 
the argument of $Z_{\sigma_m}(E)$ for $m\to \infty$
goes, 
(modulo $2\pi$, )
\begin{align}\label{infty}
\arg Z_{\sigma_m}(E) \lr \left\{ \begin{array}{cc}
\pi & \quad \dim \Supp(E)=0, \\
\frac{\pi}{2} & \quad \dim \Supp(E)=1, \\
0 & \quad \dim \Supp(E)=2, \\
-\frac{\pi}{2} & 
\quad \dim \Supp(E)=3.
\end{array}\right. \end{align}
Since the category $\aA^p$ is generated by 
$\Coh_{\le 1}(X)$ and $\Coh_{\ge 2}(X)[1]$, the above 
asymptotic behavior of $Z_{\sigma_m}(E)$ shows the result. 
\end{proof}
Given $\sigma \in A(X)_{\mathbb{C}}$ and 
a non-zero object $E\in \aA^p$, we can uniquely 
determine the phase of $Z_{\sigma_m}(E)$ by 
$$\phi_{\sigma_m}(E)=\frac{1}{\pi}\Imm \log Z_{\sigma_m}(E) 
\in \left(\frac{1}{4}, \frac{5}{4}\right),$$
for $m\gg 0$. 
For non-zero $F, E\in \aA^p$, 
we simply write 
$$\phi_{\sigma}(F) \prec \phi_{\sigma}(E), \quad 
\phi_{\sigma}(F) \preceq  \phi_{\sigma}(E), $$
if $\phi_{\sigma_m}(F) < \phi_{\sigma_m}(E)$, 
$\phi_{\sigma_m}(F)\le \phi_{\sigma_m}(E)$ for 
$m\gg 0$ respectively. 
Below we introduce the notion of 
limit (semi)stable objects. 
\begin{defi}\emph{
For $\sigma \in A(X)_{\mathbb{C}}$, 
a non-zero object $E\in \aA^p$ is called $\sigma$-\textit{limit stable} 
(resp. $\sigma$-\textit{limit semistable})
if for any non-zero subobject $F\subsetneq E$, one has 
$$\phi_{\sigma}(F) \prec \phi_{\sigma}(E), \quad (\mbox{resp. } 
\phi_{\sigma}(F) \preceq \phi_{\sigma}(E).) $$}
\end{defi}
\begin{rmk}\label{rmk:induce}
\emph{
In Lemma~\ref{lem:asy}, the smallest $m>0$ for which 
(\ref{asy}) holds depends on $E$, the function 
$Z_{\sigma_m}$ does not give 
stability functions on $\aA^p$ for any $m$.
On the other hand, the function $Z_{\sigma}$ induces 
the stability condition on the subcategory $\Coh_{\le 1}(X)\subset \aA^p$
by the composition, 
$$K(\Coh_{\le 1}(X)) \lr K(X) \stackrel{Z_{\sigma}}{\lr} \mathbb{C}.$$
The induced stability condition is the same one constructed 
in~\cite[Lemma 3.4]{ToBPS}.  
}
\end{rmk}
\begin{rmk}\emph{Our notion of 
limit stability is included in the notion of polynomial 
stability introduced by A.~Bayer~\cite{Bay} independently. 
Some of the results in this section, especially 
Theorem~\ref{property} (i), are proved in~\cite{Bay} in more general 
setting, although the proofs are different. 
}
\end{rmk}
\begin{rmk}\emph{
It is easy to see some standard stability properties for 
limit stability. For example, let $E, F \in \aA^p$ be 
$\sigma$-limit semistable with $\phi_{\sigma}(E)\succ \phi_{\sigma}(F)$. 
Then $\Hom(E, F)=0$. Also for $\sigma$-limit stable 
object $E\in \aA^p$, one has $\Hom(E, E)=\mathbb{C}$. 
}
\end{rmk}

In the following, we give some examples of limit (semi)stable 
objects. 
The proofs are straightforward and we leave them to the readers. 
\begin{exam}\label{exam:lim}\emph{
(i) Let $F$ be a $\mu$-stable vector bundle
on $X$. Then $F[1] \in \aA^p$ 
and it is $\sigma$-limit stable for any $\sigma \in A(X)_{\mathbb{C}}$.} 

\emph{(ii) Let us take $\sigma=B+i\omega \in A(X)_{\mathbb{C}}$
and 
$F\in \Coh_{\le 1}(X)\subset \aA^p$. 
Then noting Remark~\ref{rmk:art}
and Remark~\ref{rmk:induce}, we can easily see
that $F$ is a $\sigma$-limit
semistable if and only if $F$ is $(B, \omega)$-twisted 
semistable sheaf, i.e. 
for any non-zero subsheaf $F'\subset F$, one has 
$\mu_{\sigma}(F') \le \mu_{\sigma}(F)$, where 
$$\mu_{\sigma}(F)=
\frac{\ch_3(F)-B\ch_2(F)}{\omega \ch_2(F)}\in \mathbb{R}.$$
 }

\emph{(iii) Let $x\in X$ be a closed point and 
$I_x \subset \oO_X$ the ideal sheaf. Then 
$I_x$ is a Gieseker stable sheaf, but $I_x[1]\in \aA^p$ is not 
$\sigma$-limit semistable. In fact we have the
exact sequence in $\aA^p$, 
$$0 \lr \oO_x \lr I_x[1] \lr \oO_X[1] \lr 0,$$
with $\phi_{\sigma}(\oO_x)\succ \phi_{\sigma}(I_x[1])$, 
which destabilizes $I_x[1]$.} 
\end{exam}
For objects in $\aA_i^p$, we have the following lemma. 
\begin{lem}\label{infty}
For a non-zero object $E\in \aA_i^p$ (i=1, 1/2,) 
we have 
$$\phi_{\sigma_m}(E) \to i, \quad \mbox{ for }m\to \infty.$$
\end{lem}
\begin{proof}
For $E=F[1]$ or $E=\oO_x$, where $F$ is a pure two 
dimensional sheaf and $x\in X$ is a closed point, 
the result follows by the formula (\ref{formula1}). 
By the definition of $\aA_1^p$, the result 
follows for any $E\in \aA_1^p$. 
Next let us take a non-zero object $E\in \aA_{1/2}^p$. 
Then by the definition of $\aA_{1/2}^p$, 
we have either $\hH^{-1}(E)$ is a torsion free 
sheaf, or $\hH^{-1}(E)=0$ and $\hH^{0}(E)$ is a pure 
one dimensional sheaf. In both cases, 
the result follows by the formula (\ref{formula1}). 
\end{proof}
 We 
have the following characterization of limit stable objects.  
\begin{lem}\label{iff2}
An object $E\in \aA^p$ is $\sigma$-limit (semi)stable 
with $\phi_{\sigma_m}(E) \to i$ for $m\to \infty$
if and only if $E\in \aA^p_i$ and for any strict monomorphism  
$0\neq F\hookrightarrow E$ in $\aA_{i}^p$,  one has 
$\phi_{\sigma}(F) \prec \phi_{\sigma}(E)$. 
(resp. $\phi_{\sigma}(F) \preceq \phi_{\sigma}(E)$.)
\end{lem}
\begin{proof}
Suppose first that $E$ is $\sigma$-limit semistable. 
By Lemma~\ref{tor}, there is an exact sequence, 
$$0 \lr E_{1} \lr E \lr E_{1/2} \lr 0, $$
in $\aA^p$ with $E_i \in \aA_i$. 
By Lemma~\ref{infty},
the limit semistability of $E$ implies 
$E_1=0$ or $E_{1/2}=0$. Hence 
if $\phi_{\sigma_m}(E)$ goes to $i$, we have 
$E\in \aA_{i}^p$.
Next   
assume that $E\in \aA^p_1$ and consider an 
exact sequence in $\aA^p$, 
$$0 \lr F \lr E \lr G \lr 0.$$
By Lemma~\ref{tor}, there
 is an exact sequence $0 \to F_1 \to F \to F_{1/2} \to 0$
with $F_i \in \aA^p_i$. 
Lemma~\ref{infty} yields, 
$$\phi_{\sigma}(F_1)\succeq \phi_{\sigma}(F) \succeq \phi_{\sigma}(F_{1/2}).$$ 
Composing the injections 
$F_1 \hookrightarrow F \hookrightarrow E$, we obtain
 the exact sequence in $\aA^p$, 
$$0 \lr F_1 \lr E \lr G' \lr 0.$$
For any $F' \in \aA_{1/2}^p$, 
we have $\Hom(G', F') \subset \Hom(E, F')=0$. 
Therefore $G'\in \aA^p_1$, 
i.e. $F_1 \hookrightarrow E$ is
a strict monomorphism. 
 Hence if $\phi_{\sigma}(F_1) \preceq \phi_{\sigma}(E)$ holds, then
$\phi_{\sigma}(F) \preceq \phi_{\sigma}(E)$, hence
$E$ is $\sigma$-limit semistable. 
The proofs for limit stable objects and the case of
$i=1/2$ are similar and we leave them to
the reader. 
\end{proof}
For $\sigma=B+i\omega \in A(X)_{\mathbb{C}}$, 
let $\sigma^{\vee}=-B+i\omega$. 
Combining Lemma~\ref{dual} with Lemma~\ref{iff2}, 
we have the following compatibility of limit
stability with the dualizing functor. 
\begin{lem}\label{compati}
We have the following. 
\begin{align*}
 E\in \aA^p_1 \mbox{ is }\sigma\mbox{-limit (semi)stable }
 & \Leftrightarrow \mathbb{D}(E)[1] \in \aA^p_1
 \mbox{ is }\sigma^{\vee}\mbox{-limit (semi)stable}, \\
 E\in \aA^p_{1/2}\mbox{ is }\sigma\mbox{-limit (semi)stable } & 
 \Leftrightarrow \mathbb{D}(E) \in \aA^p_{1/2}\mbox{ is }\sigma^{\vee}\mbox{-limit (semi)stable.}
\end{align*}
\end{lem}
\begin{proof}
For $v\in H^{\rm{even}}(X, \mathbb{R})$, let $v^{\vee}$ be the dual operator, 
$$v=(v_0, v_1, v_2, v_3) \longmapsto v^{\vee}=(v_0, -v_1, v_2, -v_3).$$
Here $v_i$ is the $H^{2i}$-component of $v$. 
Then we have $v^{-B}(\mathbb{D}(E))=v^{B}(E)^{\vee}$, hence
\begin{align*}
Z_{\sigma^{\vee}}(\mathbb{D}(E)) &= -\int e^{-mi\omega}v^B(E)^{\vee}, \\
&= -\overline{Z_{\sigma}(E)}.
\end{align*}
Therefore the result follows from Lemma~\ref{dual} with Lemma~\ref{iff2}
directly. 
\end{proof}
Finally in this section, we prove the following theorem. 
\begin{thm}\label{property}
For $\sigma \in A(X)_{\mathbb{C}}$, we have the following. 

(i) For a non-zero $E\in \aA^p$, there exists a filtration in $\aA^p$,
\begin{align}\label{Harder} E_0 \subset E_1 \subset \cdots \subset E_n =E, 
\end{align}
such that each $F_i=E_i/E_{i+1}$ is $\sigma$-limit 
semistable with $\phi_{\sigma}(F_i)\succ \phi_{\sigma}(F_{i+1})$. i.e. 
(\ref{Harder}) is a Harder-Narasimhan filtration. 

(ii) For a $\sigma$-limit semistable object $E\in \aA^p$, there exists a 
filtration in $\aA^p$, 
\begin{align}\label{Jordan} E_0 \subset E_1 \subset \cdots \subset E_n =E, 
\end{align}
such that each $F_i=E_i/E_{i+1}$ is $\sigma$-limit 
stable with $\phi_{\sigma}(F_i)=\phi_{\sigma}(F_{i+1})$. i.e. 
(\ref{Jordan}) is a Jordan-H\"older filtration. 
\end{thm}
\begin{proof}
(i) In Proposition~\ref{suff}, let us replace 
$\aA$, $\phi$, ``inclusions'', ``surjections'', by 
$\aA_i^p$, $\phi_{\sigma}$, 
``strict monomorphisms'', ``strict epimorphisms'',
respectively. 
As Lemma~\ref{finlen} provides the corresponding 
sufficient condition, 
we can follow
  the same proof of Proposition~\ref{suff}
 in~\cite[Proposition 2.4]{Brs1},
and show the following. 
For any $E\in \aA_i^p$, there
 is a finite sequence of strict monomorphisms in $\aA_i^p$,  
\begin{align}\label{Harder2} E_0 \hookrightarrow
 E_1 \hookrightarrow \cdots \hookrightarrow E_n =E, 
\end{align}
such that for any strict monomorphism 
$F\hookrightarrow F_i=E_i/E_{i+1} \in \aA_i^p$,  
one has $\phi_{\sigma}(F) \preceq \phi_{\sigma}(F_i)$, 
and $\phi_{\sigma}(F_i)\succ \phi_{\sigma}(F_{i+1})$. 
By Lemma~\ref{iff2}, $F_i$ is a $\sigma$-limit semistable object,
hence the filtration (\ref{Harder2}) gives the Harder-Narasimhan filtration. 

Let us take an object $E\in \aA^p$. We have an exact sequence, 
$$0 \lr E_{1} \lr E \lr E_{1/2} \lr 0, $$
with $E_i \in \aA_i^p$. Composing the Harder-Narasimhan filtrations of 
$E_1$, $E_{1/2}$, we obtain the Harder-Narasimhan filtration of $E$. 

(ii) Since any $\sigma$-limit semistable object is 
contained in $\aA_1^p$ or $\aA_{1/2}^p$ by Lemma~\ref{iff2}, 
the result follows from Lemma~\ref{finlen}. 
\end{proof}

\begin{rmk}\emph{
The existence of a Harder-Narasimhan filtration is 
guaranteed once we show that 
there are no infinite sequences such as (\ref{inf1}), (\ref{inf2}) 
for limit stability. 
Unfortunately this is not true. 
In fact in the notation of Remark~\ref{rmk:noe}, 
we have the following
infinite sequence, 
$$\oO_X \oplus \oO_H[1] \twoheadrightarrow \oO_X \oplus \oO_H(C)[1] 
\twoheadrightarrow \oO_X \oplus \oO_H(2C)[1] \twoheadrightarrow \cdots,$$
which satisfies that for $\sigma=i\omega$,  
$$\phi_{\sigma}(\oO_X \oplus \oO_H[1])\succ
\phi_{\sigma}(\oO_X \oplus \oO_H(C)[1])\succ
\phi_{\sigma}(\oO_X \oplus \oO_H(2C)[1])\succ \cdots.$$}
\end{rmk}

\section{Moduli spaces of limit stable objects}\label{section:moduli}
For a Calabi-Yau 3-fold $X$, let us take elements, 
$$\beta \in H^4(X, \mathbb{Q}), \quad n\in H^6(X, \mathbb{Q})\cong 
\mathbb{Q}.$$
This section is devoted to study the moduli problem of 
limit stable objects $E\in \aA^p$, satisfying 
$\det E=\oO_X$ and the following numerical condition, 
\begin{align}\label{chern}
(\ch_0(E), \ch_1(E), \ch_2(E), \ch_3(E))=(-1, 0, \beta, n).
\end{align}
Note that if $E\in \aA^p$ is limit stable satisfying 
(\ref{chern}), then $E\in \aA_{1/2}^p$ by Lemma~\ref{iff2}.
For $\sigma \in A(X)_{\mathbb{C}}$, 
 let $L_n(X, \beta)$, $L_n^{\sigma}(X, \beta)$ be the sets objects, 
\begin{align}\label{L1}
L_n(X, \beta) &\cneq \{ E\in \aA_{1/2}^p \mid \det E=\oO_X 
\mbox{ and }\ch(E)\mbox{ satisfies (\ref{chern}) }\}, \\
\label{L2}
L_n^{\sigma}(X, \beta) &\cneq 
\{ E\in L_n(X, \beta) \mid E \mbox{ is }\sigma\mbox{-limit stable }\}.
\end{align}
\begin{rmk}\label{quasi-iso}\emph{
If $E$ is quasi-isomorphic to a two term 
complex $(\oO_X \stackrel{s}{\to} F)$,
where $F$ is a pure one dimensional sheaf
located in degree zero, and 
$$\ch_2(F)=\beta, \quad \ch_3(F)=n,$$
then $E\in L_n(X, \beta)$.}
\end{rmk}
From this section, we use the following notation. 
For a relatively perfect object (cf.~\cite[Definition 2.1.1]{LIE})
$\eE \in D^b(X\times S)$
and a morphism $T\to S$, we denote by 
$\eE_T \in D^b(X\times T)$ the derived pull-back of $\eE$. 
The moduli problem of objects in the derived category has
been studied in some articles, see~\cite{Inaba}, \cite{LIE}, \cite{Tst3}. 
In this paper, we use the 
algebraic space constructed by Inaba~\cite{Inaba}, 
which provides a ``mother space'' of our moduli problem. 
 Let
$\mM$ be the functor, 
$$\mM \colon (\Sch/\mathbb{C}) \lr (\mathrm{Set}), $$
which sends a $\mathbb{C}$-scheme $S$ to a family of 
simple complexes $\eE \in D^b(X\times S)$,
(up to isomorphism,) where 
an object $E\in D^b(X)$ is called a \textit{simple complex} if 
\begin{align}\label{simple}
\Hom(E, E)=\mathbb{C}, \quad \Ext^{-1}(E, E)=0.\end{align}
Then Inaba~\cite{Inaba} shows that the \'{e}tale sheafication 
of $\mM$, denoted by $\mM^{\rm{et}}$, is an algebraic space
of locally finite type. Let $\mM_0^{\rm{et}}$ be the 
closed fiber at $[\oO_X]\in \Pic(X)$ 
with respect to the following morphism, 
$$\det \colon \mM^{\rm{et}} \ni E \longmapsto \det E \in \Pic(X).$$
Since any object $E\in L_n^{\sigma}(X, \beta)$ 
satisfies (\ref{simple}), 
there is a subfunctor
\begin{align}\label{subspace}
\lL_n^{\sigma}(X, \beta) \subset \mM_0^{\rm{et}},\end{align}
whose $S$-valued point consists of $\eE \in \mM_0^{\rm{et}}(S)$ with 
$\eE_s \in L_n^{\sigma}(X, \beta)$ for any $s\in S$.  
Our purpose in this section is to show that 
$\lL_n^{\sigma}(X, \beta)$ is an algebraic subspace of 
$\mM_0^{\rm{et}}$. 
We use the same strategy as in~\cite{Tst3}, namely we show that 
(\ref{subspace}) is an open immersion and 
$L_n^{\sigma}(X, \beta)$ is bounded. 

\subsection{Characterizations of limit stable objects}
In this paragraph,
 we give some characterizations for 
objects in $L_n(X, \beta)$ to be limit stable. 
First we show the following. 
\begin{lem}\label{lem:C}
For an object $E\in L_n(X, \beta)$, there is a subscheme $C\subset X$ 
with $\oO_C$ a pure one dimensional sheaf (or zero) such that
$\hH^{-1}(E)$ is isomorphic to the ideal sheaf $I_C\subset \oO_X$.
\end{lem}
\begin{proof}
Since $E\in \aA_{1/2}^p$,
 $\hH^{-1}(E)$ is a torsion free sheaf of rank one with 
trivial determinant. 
We have the injection, 
$$\hH^{-1}(E) \hookrightarrow \hH^{-1}(E)^{\vee \vee}\cong \oO_X, $$
which shows $\hH^{-1}(E)\cong I_C$ for a subscheme $C\subset X$.  
By the condition (\ref{chern}), we have $\dim C\le 1$. 
Also if $\oO_C$ contains a zero dimensional 
subsheaf, 
there is $x\in X$ with injections in $\aA^p$,  
$$\oO_x \hookrightarrow \oO_C \hookrightarrow I_C[1] \hookrightarrow E.$$
Here $\oO_C \hookrightarrow I_C[1]$ corresponds to the 
extension $0 \to I_C \to \oO_X \to \oO_C \to 0$. 
Since $E\in \aA_{1/2}^p$, this is a contradiction. 
\end{proof}

\begin{rmk}\label{turnout}
\emph{
For $E\in L_n(X, \beta)$, 
Lemma~\ref{lem:C} yields, 
\begin{align}\label{rmk:bou}
\beta &=\ch_2(\oO_C) +\ch_2(\hH^0(E)) \in H^4(X, \mathbb{Z}), \\
\notag 
n &=\ch_3(\oO_C)+\ch_3(\hH^0(E)) \in H^6(X, \mathbb{Z}) \cong \mathbb{Z}.
\end{align}
Hence below we always assume $\beta \in H^4(X, \mathbb{Z})$,
$n\in \mathbb{Z}$, and 
$\beta$ is an effective class, i.e. 
the Poincar\'e dual of the homology class of an effective one 
cycle on $X$.} 
\end{rmk}

Next we show the following. 

\begin{lem}\label{prop:char}
For $\sigma \in A(X)_{\mathbb{C}}$, an object $E\in L_n(X, \beta)$ 
is $\sigma$-limit stable if and only if the following conditions hold. 

(a) For any pure one dimensional sheaf $G\neq 0$
which admits a strict epimorphism $E\twoheadrightarrow G$ in 
$\aA_{1/2}^p$, one 
has $\phi_{\sigma}(E) \prec \phi_{\sigma}(G)$.

(b) For any pure one dimensional sheaf $F\neq 0$
which admits a strict monomorphism $F\hookrightarrow E$ in 
$\aA_{1/2}^p$, one 
has $\phi_{\sigma}(F) \prec \phi_{\sigma}(E)$.
\end{lem}
\begin{proof}

For a $\sigma$-limit stable object $E\in L_n(X, \beta)$, the 
conditions $(a)$, $(b)$ follow from the definition of limit stability. 

Next suppose that $E\in L_n(X, \beta)$ satisfies $(a)$ and $(b)$.
Applying Lemma~\ref{iff2}, it is enough to show that for any
non-trivial exact sequence in 
$\aA_{1/2}^p$, 
$$0 \lr F \lr E \lr G \lr 0, $$
we have $\phi_{\sigma}(F) \prec \phi_{\sigma}(E)$. 
If $\hH^{-1}(F)=0$, then $F$ is a pure one dimensional sheaf
and $\phi_{\sigma}(F) \prec \phi_{\sigma}(E)$ follows 
from $(b)$. If $\hH^{-1}(F) \neq 0$, then it is 
a torsion free sheaf of rank one by Lemma~\ref{lem:C}. It follows 
that $\hH^{-1}(G)$ is a torsion sheaf, hence zero 
because of $G\in \aA_{1/2}^p$. So $G$
is a pure one dimensional sheaf,
and we obtain $\phi_{\sigma}(E)\prec \phi_{\sigma}(G)$ by $(a)$.  
\end{proof}

\begin{rmk}\label{dual2}\emph{
By Lemma~\ref{dual} and the same argument as in Lemma~\ref{compati},
 the condition $(b)$ 
of Lemma~\ref{prop:char}
can be replaced by the following. 
For any pure one dimensional sheaf $G'\neq 0$ 
which admits a strict epimorphism 
$\mathbb{D}(E) \twoheadrightarrow G'$ in $\aA_{1/2}^p$, one has 
$\phi_{\sigma^{\vee}}(\mathbb{D}(E)) \prec \phi_{\sigma^{\vee}}(G')$.} 
\end{rmk}

\begin{rmk}\label{rmk:eq}\emph{
Since $E\in \aA_{1/2}^p$ is concentrated on $[-1,0]$, 
giving a strict epimorphism $E\twoheadrightarrow G$ 
as in $(a)$ 
of Lemma~\ref{prop:char} 
is equivalent to giving a surjection of sheaves 
$\hH^0(E) \twoheadrightarrow G$.}
\end{rmk}

As for strict monomorphism $F \hookrightarrow E$ 
in $(b)$ of Lemma~\ref{prop:char}, we have the following. 
\begin{lem}\label{asfor}
Let $F\hookrightarrow E$ be as in $(b)$ of Lemma~\ref{prop:char}, 
and $C\subset X$ as in Lemma~\ref{lem:C}.  
Then there are subsheaves, 
$$F_1 \subset \oO_C, \quad F_2 \subset \hH^0(E), $$
such that $F$ is written as an extension, 
$$0 \lr F_1 \lr F \lr F_2 \lr 0.$$
\end{lem}
\begin{proof}
Let 
$$0 \lr F \lr E \lr G \lr 0,$$
be the exact sequence in $\aA^p$. Taking cohomology, 
we obtain the exact sequences of sheaves, 
\begin{align}
\label{ex1}& 0 \lr \hH^{-1}(E) \lr \hH^{-1}(G) \lr F_1 \lr 0, \\
\label{ex2}& 0 \lr F_1 \lr F \lr F_2 \lr 0, \\
\label{ex3}& 0 \lr F_2 \lr \hH^0(E) \lr \hH^0(G) \lr 0.
\end{align}
Since $\hH^{-1}(E)$ is torsion free and $F_1 \in \Coh_{\le 1}(X)$, 
we have 
$$F_1 \subset \hH^{-1}(E)^{\vee \vee}/\hH^{-1}(E)\cong \oO_C,$$
from the sequence (\ref{ex1}). Therefore the sequence (\ref{ex2}) 
gives the desired extension. 
\end{proof}

In Lemma~\ref{prop:char}, let us write the 
condition $\phi_{\sigma}(F) \prec \phi_{\sigma}(E)$ in a simpler way. 
For $F \in \Coh_{\le 1}(X)$, let $\mu_{\sigma}(F) \in \mathbb{R}$ 
be as in Example~\ref{exam:lim} (ii). 
\begin{lem}\label{equ}
For $\sigma=B+i\omega \in A(X)_{\mathbb{C}}$, 
 $E\in L_n(X, \beta)$ and $F\in \Coh_{\le 1}(X)$, we have 
$\phi_{\sigma}(F) \prec \phi_{\sigma}(E)$, 
(resp. $\phi_{\sigma}(F)\succ \phi_{\sigma}(E)$,)
if and only 
if one of the following conditions hold. 
\begin{itemize}
\item We have the following inequality,  
\begin{align}\label{ineq1}
\mu_{\sigma}(F)<-\frac{3B\omega^2}{\omega ^3},
\quad (\mbox{resp. }\mu_{\sigma}(F)>-\frac{3B\omega^2}{\omega ^3}.)\end{align}
\item We have $\mu_{\sigma}(F)=-3B \omega^2/\omega^3$ and 
\begin{align}\label{ineq2}
\omega v_2^B(E) \mu_{\sigma}(F)< v_3^B(E), \quad 
(\mbox{resp. }\omega v_2^B(E) \mu_{\sigma}(F)> v_3^B(E).)\end{align}
\end{itemize}
\end{lem}
\begin{proof}
The condition $\phi_{\sigma}(F) \prec \phi_{\sigma}(E)$ is 
equivalent to 
\begin{align}\label{equi}
\frac{\Ree Z_{\sigma_m}(F)}{\Imm Z_{\sigma_m}(F)}> 
\frac{\Ree Z_{\sigma_m}(E)}{\Imm Z_{\sigma_m}(E)}, \end{align}
for $m\gg 0$. 
Since we have 
\begin{align*}
v^B(F) &=(0, 0, \ch_2(F), \ch_3(F)-B\ch_2(F)), \\
v^B(E) &=(-1, B, v_2^B(F), v_3^B(F)), 
\end{align*}
the inequality (\ref{equi}) is equivalent to 
$$-\frac{\mu_{\sigma}(F)}{m}
=\frac{-\ch_3(F)+B\ch_2(F)}{m\omega \ch_2(F)}
>
\frac{m^2 \omega^2 B/2 -v_3^B(E)}{m^3\omega^3/6 +m\omega v_2^B(E)},
$$
for $m\gg 0$ by the formula (\ref{formula}). The above 
inequality is equivalent to 
$$\frac{1}{6}m^2 \omega^3 \left( 
\mu_{\sigma}(F)+\frac{3\omega^2 B}{\omega^3} \right) 
< -\omega v_2^{B}(E)\mu_{\sigma}(F)+v_3^{B}(E),$$
for $m\gg 0$. Therefore (\ref{ineq1}) or (\ref{ineq2})
must be satisfied. 
\end{proof}

\subsection{Evaluations of numerical classes}\label{Nbeta}
In this paragraph, we evaluate the numerical classes of 
$\hH^{-1}(E)$, $\hH^{0}(E)$ for $E\in L_n^{\sigma}(X, \beta)$. 
Below we fix an ample divisor $H$ on $X$,
and set
\begin{align}\label{N(beta)}
\nN(\beta)\cneq \{ \beta' \in H^4(X, \mathbb{Z}) \mid 
\beta' \mbox{ is an effective class with } 0 \le \beta' \cdot H \le 
\beta \cdot H\}.
\end{align}
The following Lemma~\ref{lem:fin} and Lemma~\ref{-infty}
 seem well-known, but 
we give the proof for the reader's convenience. 
\begin{lem}\label{lem:fin}
The set $\nN(\beta)$ is a finite set. 
\end{lem}
\begin{proof}
For any ample divisor $H'$ on $X$, we have 
$$\nN(\beta) \subset \{ \beta' \in H^4(X, \mathbb{R}) \mid 
\beta'\cdot H' \ge 0\}.$$
Since the ample cone is an open cone, one can find a 
compact convex polytope in $H^4(X, \mathbb{R})$
 which contains $\nN(\beta)$. 
Therefore $\nN(\beta)$ is a finite set. 
\end{proof}
Next we set $m(\beta)\in [-\infty, \infty)$ as follows, 
$$m(\beta)\cneq \mathrm{inf} \{ \ch_3(\oO_C) \mid 
C\subset X \mbox{ \rm{satisfies} }\dim C=1, [C] \in \nN(\beta)\}.$$
\begin{lem}\label{-infty}
We have $m(\beta)>-\infty$. 
\end{lem}
\begin{proof}
Let $\Hilb_n(X, \beta)$ be the Hilbert scheme
of one dimensional subschemes $C\subset X$ with 
$$\beta=[C], \quad n=\ch_3(\oO_C).$$
If $\Hilb_{n-k}(X, \beta)$ is non-empty for $k>0$, 
then we have 
$$\dim \Hilb_n(X, \beta)\ge 3k, $$
by adding $k$-floating points to a subscheme 
$C'\subset X$ with $[C']=\beta$, $n-k=\ch_3(\oO_{C'})$. 
Then the boundedness of $\Hilb_n(X, \beta)$ implies 
that $\Hilb_{n-k}(X, \beta)=\emptyset$ for $k\gg 0$. 
\end{proof}
Finally we show the following lemma. 
\begin{lem}\label{lem:final}
(i) The image of the map, 
$$L_n(X, \beta) \ni E \longmapsto (\ch_2(\hH^{-1}(E)), \ch_2(\hH^0(E)))
\in H^4(X, \mathbb{Z})^{\oplus 2},$$
is a finite set. 

(ii) For $E\in L_n(X, \beta)$, let $\beta'=-\ch_2(\hH^{-1}(E))$. 
Then we have, 
$$\ch_3(\hH^0(E)) \le n-m(\beta') \le n-m(\beta).$$
Moreover for $\sigma =B+i\omega
\in A(X)_{\mathbb{C}}$, the image of the map, 
$$L_n^{\sigma}(X, \beta) \ni E \longmapsto
 (\ch_3(\hH^{-1}(E)), \ch_3(\hH^0(E)))
\in \mathbb{Z}^{\oplus 2},$$
is a finite set. 
\end{lem}
\begin{proof}
(i) For $E\in L_n^{\sigma}(X, \beta)$, the equality (\ref{rmk:bou})
implies, 
$$(\ch_2(\hH^{-1}(E)), \ch_2(\hH^0(E))) \in \nN(\beta)\times \nN(\beta).$$
Hence Lemma~\ref{lem:fin} yields the result. 

(ii) For $E\in L_n(X, \beta)$, 
we have $\hH^{-1}(E)=I_C$ where $C\subset X$ is as in 
Lemma~\ref{lem:C}.
Since $[C]=\beta'$, we have 
\begin{align*}
\ch_3(\hH^0(E)) &=n-\ch_3(\oO_C) \\
& \le n-m(\beta').
\end{align*}
Also since $\beta' \in \nN(\beta)$, we have 
$n-m(\beta') \le n-m(\beta)$. 
Suppose that $E$ is $\sigma$-limit stable. 
If $\hH^0(E)$ is non-zero, Lemma~\ref{prop:char}, 
Remark~\ref{rmk:eq} and Lemma~\ref{equ}
show, 
$$-\frac{3B\omega^2}{\omega^3}\le \mu_{\sigma}(\hH^0(E)) 
=\frac{\ch_3(\hH^0(E))-B\ch_2(\hH^0(E))}{\omega \ch_2(\hH^0(E))}.$$
Since $\ch_2(\hH^0(E)) \in \nN(\beta)$, 
 the value $\ch_3(\hH^0(E))\in \mathbb{Z}$
  is also bounded below by the above inequality and
  Lemma~\ref{lem:fin}. 
\end{proof}

\subsection{Boundedness of limit stable objects}
In this paragraph, we show the 
boundedness of limit stable objects, which is relevant 
to the existence of the moduli space.  Recall that 
a set of objects $\sS \subset D^b(X)$ is 
\textit{bounded} if there is a finite type $\mathbb{C}$-scheme 
$Q$ and an object $\eE \in D^b(X\times Q)$ such that 
any object $E\in \sS$ is isomorphic to $\eE_q$ for 
some $q\in Q$. We first show the boundedness
of some subsets of objects in $\Coh_{\le 1}(X)$. 
For $\beta \in H^4(X, \mathbb{Z})$ and $n\in \mathbb{Z}$, 
we set 
$$S_n(X, \beta)\cneq \{ E\in \Coh_{\le 1}(X) \mid 
\ch_2(E)=\beta, \ch_3(E)=n\}.$$
Also let us fix $\sigma =B+i\omega
\in A(X)_{\mathbb{C}}$ and $\mu \in \mathbb{R}$. 
\begin{lem}\label{lem:set of objects}
(i) The following 
set of objects is bounded.
\begin{align*}
S_n(X, \beta, \sigma, \mu)=
\left\{ E\in S_n(X, \beta) : \begin{array}{l}
 \mu_{\sigma}(G) \ge \mu \mbox{ \rm{for any surjection} } \\
 E \twoheadrightarrow G \mbox{ \rm{in} } \Coh_{\le 1}(X) 
 \end{array}
   \right\}.
\end{align*}
(ii) The following 
set of objects is bounded. 
\begin{align*}
S_n'(X, \beta, \sigma, \mu)=
\left\{ G\in \Coh_{\le 1}(X) \cap \aA_{1/2}^p
: \begin{array}{l}
\mbox{\rm{there is} }E\in S_n(X, \beta) \mbox{ \rm{and a surjection} }  \\
E \twoheadrightarrow G \mbox{ in }\Coh_{\le 1}(X)
\mbox{ \rm{and} } \mu_{\sigma}(G) \le \mu
 \end{array}
   \right\}.
\end{align*}
\end{lem}
\begin{proof}
(i) For $E\in S_n(X, \beta, \sigma, \mu)$, let
$$0=E_0 \subset E_1 \subset \cdots \subset E_l=E
$$ be the Harder-Narasimhan filtration with respect to  
$(B, \omega)$-twisted semistablity. (Or equivalently
Harder-Narasimhan 
filtration with respect to the induced stability 
condition $(\Coh_{\le 1}(X), Z_{\sigma})$.) 
Note that for $F_i=E_i/E_{i-1}$, we have 
$\ch_2(F_i) \in \nN(\beta)$, and the function which 
sends $E \in S_n(X, \beta, \sigma, \mu)$ to the 
number of $(B, \omega)$-twisted semistable factors $l$ is bounded. 
By the definition of $S_n(X, \beta, \sigma, \mu)$, we see 
$$\mu_{\sigma}(E) \ge \mu_{\sigma}(E/E_i) \ge \mu, $$
for any $i$. Therefore for each $i$, 
we have only finite number of possibilities for the pair,
$$(\ch(E_i), \ch(E/E_i)) \in H^{\ast}(X, \mathbb{Q})^{\oplus 2}. $$
Hence the possibilities for $\ch(F_i) \in H^{\ast}(X, \mathbb{Q})$
is also bounded. On the other hand, the set of 
$(B, \omega)$-twisted semistable sheaves $F\in \Coh_{\le 1}(X)$ 
with a fixed numerical class is bounded. 
(See for instance~\cite[Lemma 3.8]{ToBPS}.)
Therefore $S_n(X, \beta, \sigma, \mu)$ is bounded. 

(ii) For $G\in S_n'(X, \beta, \sigma, \mu)$, 
note that $\ch_2(G) \in \nN(\beta)$. 
Hence $\mu_{\sigma}(G) \le \mu$ implies that there is $\mu' \in \mathbb{R}$, 
which depends only on $B$ and $\beta$, 
such that 
$$\mu_{i\omega}(G)=
\frac{\ch_3(G)}{\omega \ch_2(G)} \le \mu', $$
for any $G\in S_n'(X, \beta, \sigma, \mu)$.
Since $G$ is a pure sheaf, one can apply~\cite[Lemma 1.7.9]{Hu}, 
and conclude that $S_n'(X, \beta, \sigma, \mu)$ is bounded. 
\end{proof}
In the following we show the boundedness of $L_n^{\sigma}(X, \beta)$. 
\begin{prop}\label{prop:bounded}
For $\sigma=B+i\omega \in A(X)_{\mathbb{C}}$, the set of objects
$L_n^{\sigma}(X, \beta)$ is bounded.
\end{prop}
\begin{proof}
It is enough to show the boundedness of the following sets
of sheaves, 
\begin{align}
\label{set1}
&\{ \hH^{-1}(E) \mid E\in L_n^{\sigma}(X, \beta)\}, \\
\label{set2}
&\{ \hH^{0}(E) \mid E \in L_n^{\sigma}(X, \beta)\}.
\end{align}
Also by Lemma~\ref{lem:final}, 
it is enough to show the boundedness of the above sets of objects 
satisfying
\begin{align*}
(\ch_2(\hH^{-1}(E)), \ch_3(\hH^{-1}(E))) &=(-\beta', -n'), \\
(\ch_2(\hH^0(E)), \ch_3(\hH^0(E))) &=(\beta'', n''), \end{align*}
for fixed numerical classes $(\beta', n')$, $(\beta'', n'')$. 
For $E\in L_n^{\sigma}(X, \beta)$, we have 
$\hH^{-1}(E)\in I_{n'}(X, \beta')$
by Lemma~\ref{lem:C},  
in particular the set of sheaves (\ref{set1}) is bounded. 
Here $I_{n'}(X, \beta')$ is 
the Hilbert scheme
as in the proof of Lemma~\ref{-infty}. 
Also Lemma~\ref{prop:char}, Remark~\ref{rmk:eq} and Lemma~\ref{equ} show
$$\hH^0(E) \in S_{n''}(X, \beta'', \sigma, \mu), $$
where $\mu=-3B\omega^2/\omega^3$. 
Therefore by Lemma~\ref{lem:set of objects} (i), 
the set of sheaves (\ref{set2}) is 
also bounded. 
\end{proof}

\subsection{Openness of limit stability}
The purpose of this paragraph is 
to show that the embedding
$\lL_n^{\sigma}(X, \beta)\subset
\mM_0^{\rm{et}}$ given in (\ref{subspace}) 
is an open immersion, and complete the proof
that $\lL_n^{\sigma}(X, \beta)$ is an algebraic space of finite type. 
First we see the openness of objects in $\aA^{p}$ and $\aA_{i}^p$. 
\begin{lem}\label{lem:open}
For a variety $S$ and an object $\eE \in D^b(X\times S)$, 
the sets 
$$S^{\circ}=\{ s\in S \mid \eE_s \in \aA^p\},
 \quad S_{i}^{\circ}=\{ s\in S \mid \eE_s \in \aA^{p}_i\},$$
are open subsets in $S$.
\end{lem}
\begin{proof}
As in~\cite[Appendix A, Example 1]{AB}, the torsion theory 
$(\Coh_{\le 1}(X), \Coh_{\ge 2}(X))$ defines an open stack of 
torsion theories. 
Then~\cite[Theorem A.3]{AB} shows that $S^{\circ}$ is open. 
Let us show that $\sS_i^{\circ}$ is open in $S$. 
We set $\mathbb{D}(\eE)$ to be
$$\mathbb{D}(\eE)=
\dR \hH om(\eE, \oO_{X\times S}[2]).$$
Then we have $\mathbb{D}(\eE)_{s}\cong \mathbb{D}(\eE_s)$. 
By Lemma~\ref{dual}, $S_i^{\circ}$ are written as 
\begin{align*}
&S_1^{\circ}=S^{\circ}\cap \{ s\in S \mid \mathbb{D}(\eE)_s \in 
\aA^p[-1]\}, \\
&S_{1/2}^{\circ}
=S^{\circ}\cap \{ s\in S \mid \mathbb{D}(\eE)_s \in 
\aA^p\}.
\end{align*}
Therefore the openness of $S_{i}^{\circ}$ follows from 
the openness of $S^{\circ}$. 
\end{proof}
Next we show the boundedness of destabilizing objects 
for a family of objects in $\aA_{1/2}^p$. 
Suppose that $\eE \in D^b(X\times S)$ satisfies
$\eE_s \in L_n(X, \beta)$ for any $s\in S$. 
We set the sets of objects $\dD_{e}$, $\dD_{m}$
as follows. 
\begin{align*}
  \dD_{e}&=
\left\{ G\in \Coh_{\le 1}(X)\cap \aA_{1/2}^p: \begin{array}{l}
\mbox{\rm{there is} }s \in S \mbox{ \rm{and a strict epimorphism} }  \\
\eE_s \twoheadrightarrow G \mbox{ \rm{in} } \aA_{1/2}^p
\mbox{ \rm{with} } \phi_{\sigma}(G) \preceq \phi_{\sigma}(\eE_s)
 \end{array}
   \right\}. \\
  \dD_{m}&=
\left\{ F\in \Coh_{\le 1}(X)\cap \aA_{1/2}^p: \begin{array}{l}
\mbox{\rm{there is} }s \in S \mbox{ \rm{and a strict monomorphism} }  \\
F \hookrightarrow \eE_s \mbox{ \rm{in} } \aA_{1/2}^p
\mbox{ \rm{with} } \phi_{\sigma}(F) \succeq \phi_{\sigma}(\eE_s)
 \end{array}
   \right\}.   
 \end{align*}
We have the following. 
\begin{lem}
The sets of objects $\dD_{e}$, $\dD_{m}$ are bounded.
\end{lem}
\begin{proof}
Applying the dualizing functor $\mathbb{D}$, 
 it 
suffices to show the boundedness of $\dD_{e}$
as in Remark~\ref{dual2}. 
By taking a flattening stratification, we may 
assume that $\hH^i(\eE)$ is flat over $S$ for any $i$. 
As in Remark~\ref{rmk:eq}, any object $G\in \dD_{e}$
is obtained as a surjection, 
\begin{align}\label{obtain}
\hH^{0}(\eE_s)=\hH^0(\eE)_s \twoheadrightarrow G.\end{align}
Let $(\ch_2(\hH^0(\eE)_s), \ch_3(\hH^0(\eE)_s))=(\beta', n')$. 
By (\ref{obtain}) and Lemma~\ref{equ}, we have
$$G\in S_{n'}'(X, \beta', \sigma, \mu), $$
where $\mu=-3B\omega^2/\omega^3$. 
By Lemma~\ref{lem:set of objects} (ii),
 the set of objects $\dD_{e}$ is bounded. 
\end{proof}
Based on the above lemma, we show the following proposition. 
\begin{prop}\label{prop:construct}
(i) There exist a finite type $S$-scheme $\pi_{e}\colon 
Q_{e} \to S$, $\gG_{e} \in \Coh(X\times Q_{e})$, 
and a morphism $u_{e}\colon \eE_{Q_{e}}\to \gG_{e}$ such that 
\begin{itemize}
\item For $q\in Q_{e}$, the morphism 
$u_{{e}, q}\colon \eE_q \to \gG_{{e}, q}$ is a strict 
epimorphism in $\aA_{1/2}^p$. 
\item Any strict epimorphism $\eE_s \twoheadrightarrow G$ in $\aA_{1/2}^p$
for $G\in \dD_{e}$ is isomorphic to $u_{{e}, q}$ for some 
$q\in \pi_{e}^{-1}(s)$. 
\end{itemize}

(ii) There exist a finite type $S$-scheme $\pi_{m}\colon 
Q_{m} \to S$, $\fF_{m} \in \Coh(X\times Q_{m})$, 
and a morphism $u_{m}\colon \fF_{m} \to \eE_{Q_{m}}$ such that 
\begin{itemize}
\item For $q\in Q_{m}$, the morphism 
$u_{{m}, q}\colon \fF_{{m}, q}\to \eE_q$ is a strict 
monomorphism in $\aA_{1/2}^p$. 
\item Any strict monomorphism $F\hookrightarrow \eE_s$ in $\aA_{1/2}^p$
for $F\in \dD_{m}$ is isomorphic to $u_{{m}, q}$ for some 
$q\in \pi_{m}^{-1}(s)$. 
\end{itemize}
\end{prop}
\begin{proof}
The proof is essentially same as in~\cite[Proposition 3.17]{Tst3}, 
so we only give the outline of the construction of
 $Q_{m}$. Since $\dD_{m}$ is bounded, there 
is a $\mathbb{C}$-scheme of finite type $Q$ and 
$\fF \in \Coh(X\times Q)$, flat over $Q$, such that 
any $F\in \dD_{m}$ is isomorphic to $\fF_q$ for some $q\in Q$. 
We may assume that $\phi_{\sigma}(\fF_q) \succeq \phi_{\sigma}(\eE_s)$
and $\fF_q$ is a pure one dimensional sheaf 
for any $q\in Q$. 
 Arguing as in~\cite[Proposition 3.17]{Tst3},
there is an affine scheme of finite type $Q'$ and a morphism
$Q' \to Q\times S$ such that
\begin{itemize}
\item $Q' \to Q\times S$ is bijective on closed points. 
\item There exists a locally free sheaf $\uU$ on $Q'$ such that 
the functor 
$$(T\to Q') \longmapsto 
\hH^0(\dR q_{T\ast} \dR \hH om (\fF_T, \eE_T))\in \Coh(T),$$
is represented by the affine bundle $\mathbb{V}(\uU) \to Q'$, 
where $q_T \colon X\times T \to T$ is the projection.  
\end{itemize}
Here $\fF_T$, $\eE_T$ are obtained by the base changes
of $\fF$, $\eE$ for the 
following morphisms respectively, 
$$T \to Q' \to Q\times S \stackrel{p_1}{\to}Q, \quad 
T \to Q' \to Q\times S \stackrel{p_2}{\to}S, $$
and $p_1$, $p_2$ are projections. 
Let $u\colon
\fF_{\mathbb{V}(\uU)}\to \eE_{\mathbb{V}(\uU)}$ be 
the universal morphism and take the distinguished triangle, 
$$\fF_{\mathbb{V}(\uU)} \stackrel{u}{\lr} \eE_{\mathbb{V}(\uU)} \lr \gG.$$
Note that $u_q \colon \fF_{\mathbb{V}(\uU), q} \to \eE_{\mathbb{V}(\uU), q}$
is a strict monomorphism in $\aA_{1/2}^p$ if and only if 
$\gG_q \in \aA_{1/2}^p$. 
We construct $Q_{m}$ as 
 $$Q_{m} \cneq \{ q\in \mathbb{V}(\uU) \mid \gG_q \in \aA_{1/2}^p\}.$$ 
Then $Q_{m}$ is an open subscheme of $\mathbb{V}(\uU)$
by Lemma~\ref{lem:open}, in particular it is of finite type. 
By the construction,  
$$\pi_{m} \colon Q_{m} \to S, \quad 
\fF_{m}\cneq \fF_{\mathbb{V}(\uU)}|_{Q_{m}}
 \in \Coh(X\times Q_{m}), \quad u_m \cneq u|_{Q_m},$$
satisfy the desired property. 
 \end{proof}
 \begin{rmk}\label{desired}\emph{
 By the construction
 and Lemma~\ref{lem:C}, the object $\eE_s \in L_n(X, \beta)$ 
 is $\sigma$-limit stable if and only if 
 $$s\notin \pi_{e}(Q_e) \cup \pi_{m}(Q_m).$$
 }
 \end{rmk}
Here we collect some well-known lemmas on 
  a family of objects in $D^b(X)$.
  For the lack of reference, we also put the proofs.  
   Let $R$ be a discrete 
 valuation ring and $K$ a quotient field of $R$. 
 We denote by $t\in R$ the uniformizing parameter, 
 and 
 $o \in \Spec R$ the closed point. 
 We set $X_R=X\times \Spec R$, $X_K=X\times \Spec K$. 

 \begin{lem}\label{extension}
 Take $\fF, \eE \in D^b(X\times \Spec R)$ and a 
 non-zero morphism $f\colon \fF_K \to \eE_K$ in 
 $D^b(X_K)$. Then there is $m\in \mathbb{Z}$ such that 
 \begin{itemize}
 \item The morphism $t^m f \colon \fF_K \to \eE_K$ extends to 
 a morphism $t^m f\colon \fF \to \eE$. 
 \item The induced morphism
  $t^m f|_{X\times \{o\}} \colon \fF_{o} \to \eE_{o}$
 is non-zero.
 \end{itemize}
 \end{lem}
 \begin{proof}
 Since $\Hom(\fF, \eE)$ is a finitely generated 
 $R$-module and 
 $$\Hom_{X_R}(\fF, \eE)\otimes _R K \cong \Hom_{X_K}(\fF_K, \eE_K), $$
  there is $m\in \mathbb{Z}$ such that $t^m f$ 
 extends to $\fF \to \eE$ and $t^{m-1}f$ does not extend 
 to $\fF \to \gG$. We have the exact sequence in $\Coh(X_R)$, 
 $$0 \lr \eE \stackrel{\times t}{\lr} \eE \lr \eE_{o} \lr 0.$$
 The above sequence shows that if $t^m f|_{X\times \{o\}}$ is zero, then
 $t^{m}f$ factors though $\fF \to \eE \stackrel{\times t}{\to}\eE,$
 which gives an extension of $t^{m-1}f$. Therefore 
 $t^m f|_{X\times \{o\}}$ is non-zero.  
 \end{proof}
 Let $T$ be a (not necessary projective) smooth curve
 with a closed point $o \in T$. We set $T^{\circ}=T\setminus \{o\}$. 
 
 \begin{lem}\label{extension2}
 Suppose that $\fF \in \Coh(X\times T^{\circ})$ is flat 
 over $T^{\circ}$ and $\fF_t \in \Coh_{\le 1}(X)$ for any 
 $t\in T^{\circ}$. 
 
 (i) Assume moreover that $\fF_t$ is 
 $(B, \omega)$-twisted semistable for any $t\in T^{\circ}$.  
  Then there is $\widetilde{\fF} \in \Coh(X\times T)$, 
  which is flat over $T$, such that
  $\widetilde{\fF}|_{X\times T^{\circ}}\cong \fF$ and
  $\fF|_{X\times \{o\}}$ is also $(B, \omega)$-twisted semistable.  
  
  (ii) 
 There is an open subset $T^{'\circ}
 \subset T^{\circ}$
and a filtration of flat sheaves over $T^{'\circ}$, 
$$0=\fF_{T^{' \circ}, 0}\subset \fF_{T^{' \circ}, 1} \subset \cdots \subset 
\fF_{T^{' \circ}, n}=\fF_{T^{' \circ}}, $$
such that for any $t\in T^{' \circ}$, the induced filtration of 
$\fF_{t}$ 
is a Harder-Narasimhan filtration with respect to
$(B, \omega)$-twisted semistability.
 \end{lem} 
 \begin{proof}
  (i) 
 For $\sigma=B+i\omega$, let us consider the 
 induced stability condition $(Z_{\sigma}, \Coh_{\le 1}(X))$. 
 Using~\cite[Proposition 2.8]{Tst3}, we can assume that $B$ and $\omega$
 are defined over $\mathbb{Q}$. 
After applying some element of $\mathbb{C}$
to $\Stab(D^b(\Coh_{\le 1}(X)))$, we obtain 
a stability condition $(Z', \aA')$ on $D^b(\Coh_{\le 1}(X))$
such that $\aA'$ is a noetherian abelian category 
and $Z'(\fF_s) \in \mathbb{R}_{<0}$ for any $s\in U$. 
(See~\cite[Remark 2.7]{Tst3}.)
Then we can apply~\cite[Theorem 4.1.1]{AP}
and conclude the result. 

(ii) If $B=0$, this is shown in~\cite[Theorem 2.3.2]{Hu}.
The twisted case is similarly discussed and we leave it to the reader. 
 \end{proof}
 
 Now we are ready to show the following theorem. 
 
 \begin{thm}\label{ready}
 The embedding $\lL_n^{\sigma}(X, \beta) \subset \mM_{0}^{\rm{et}}$
 is an open immersion, and $\lL_n^{\sigma}(X, \beta)$
 is a separated algebraic space of finite type over $\mathbb{C}$. 
 \end{thm}
 \begin{proof}
 First we show that $\lL_n^{\sigma}(X, \beta) \subset \mM_{0}^{\rm{et}}$
 is an open immersion. 
 For a variety $S$, let us take a $S$-valued point,
 $\eE \in \mM_{0}^{\rm{et}}(S)$. 
 Suppose that $\eE_s \in L_n^{\sigma}(X, \beta)$ for some
 point $s\in S$. 
We want to show that there exists a Zariski open subset $s\in U \subset S$
such that $\eE_{s'} \in L_n^{\sigma}(X, \beta)$
for any $s' \in U$. 
Applying Lemma~\ref{lem:open}, 
we may assume $\eE_{s'} \in L_n(X, \beta)$ for any 
$s' \in S$. Let us construct
\begin{align*}
& \pi_e \colon Q_e \to S, \quad \gG_e \in \Coh(X\times Q_e), \quad 
u_e \colon \eE_{Q_e} \to \gG_e, \\
& \pi_m \colon Q_m \to S, \quad \fF_m \in \Coh(X\times Q_m), \quad 
u_m \colon \fF_m \to \eE_{Q_m}, 
\end{align*}
as in Proposition~\ref{prop:construct}. 
Noting Remark~\ref{desired}, 
it is enough to show 
$$s \notin \overline{\pi_{e}(Q_{e})}\cup 
\overline{\pi_{m}(Q_{m})}.$$
Let us show $s\notin \overline{\pi_{m}(Q_{m})}$. 
The proof of $s\notin \overline{\pi_{e}(Q_{e})}$
is similar. Suppose by a contradiction that 
$s\in \overline{\pi_{m}(Q_{m})}$. 
Then we can find a smooth curve $T$ with 
a closed point $o\in T$ and
a morphism 
$p\colon T \to S$ such that $p(o)=s$ and there is
a commutative diagram, 
$$\xymatrix{
T^{\circ} \ar[r]\ar[d] & Q_{m} \ar[d]^{\pi_{m}} \\
T \ar[r]^{p} & S, }$$
where $T^{\circ}=T\setminus \{o\}$.
By pulling back $\fF_{m} \in \Coh(X\times Q_{m})$
to $X\times T^{\circ}$, 
we obtain the 
object $\fF_{{m}, T^{\circ}} \in \Coh(X\times T^{\circ})$ and a 
morphism,
$$u_{{m}, T^{\circ}}\colon \fF_{{m},
 T^{\circ}} \lr \eE_{T^{\circ}}, $$
such that $\phi_{\sigma}(\fF_{{m}, t}) \succeq 
\phi_{\sigma}(\eE_{t})$ for any $t \in T^{\circ}$. 
Applying Lemma~\ref{extension2} (ii), we may assume that 
$\fF_{m, t}$ is $(B, \omega)$-twisted semistable 
for any $t\in T^{\circ}$. Then Lemma~\ref{extension2} (i)
shows that there is a flat family of $(B, \omega)$-twisted 
semistable sheaves,
$$\widetilde{\fF}_{m} \in \Coh(X \times T), $$
which extends $\fF_{{m}, T^{\circ}}$. 
Applying Lemma~\ref{extension} (i) for 
$R=\oO_{T, o}$, we obtain a non-zero morphism, 
\begin{align}\label{nonzero}
\widetilde{\fF}_{m, o} \lr \eE_s.\end{align}
Note that $\widetilde{\fF}_{m, o}$
is $\sigma$-limit semistable, $\eE_s$ is $\sigma$-limit stable, 
and $\phi_{\sigma}(\widetilde{\fF}_{m, o})\succeq \phi_{\sigma}(\eE_s)$. 
This implies that $\phi_{\sigma}(\widetilde{\fF}_{m, o})
= \phi_{\sigma}(\eE_s)$, the object $\eE_s$ is one of the 
Jordan-H\"older factors of $\widetilde{\fF}_{m, o}$,
 and the morphism (\ref{nonzero})
is surjective in $\aA^p$. 
However in this case
$\eE_s$ must be an object in $\Coh_{\le 1}(X)$
as we remarked in 
Remark~\ref{rmk:art},  
which contradicts that $\eE_s \in L_n(X, \beta)$. 
Therefore $s \notin \overline{\pi_{e}(Q_{e})}$
holds. 

Now we have proved $\lL_n^{\sigma}(X, \beta) \subset 
\mM_{0}^{\rm{et}}$ is an open immersion, hence 
$\lL_n^{\sigma}(X, \beta)$ is an algebraic space of 
locally finite type. Moreover $L_n^{\sigma}(X, \beta)$ is 
bounded by Lemma~\ref{prop:bounded}, 
which implies that $\lL_n^{\sigma}(X, \beta)$ is in 
fact of finite type. 

Finally let us show that $\lL_n^{\sigma}(X, \beta)$ 
is separated
using valuative criterion.
 Let $R$, 
 $K$, $t\in R$ and $o\in \Spec R$ be as in Lemma~\ref{extension}. 
 Take two $R$-valued points of 
$\lL_n^{\sigma}(X, \beta)$, 
$\eE_1, \eE_2 \in D^b(X_R)$. 
 Suppose that there is an isomorphism in $D^b(X_K)$, 
$$f \colon \eE_{1, K} \stackrel{\cong}{\lr} \eE_{2, K}.$$
quotient field of $R$. By the valuative criterion, it is enough 
to show that $t^m f$ extends to an isomorphism $\eE_{1} \to \eE_2$
for some $m\in \mathbb{Z}$. 
By Lemma~\ref{extension}, 
there is $m\in \mathbb{Z}$ and a morphism 
$$\widetilde{f} \colon \eE_1 \to \eE_2$$ which 
extends $t^m f$ and the induced morphism $\widetilde{f}_o
\colon \eE_{1, o}\to \eE_{2,o}$
is non-zero. Since $\eE_{1,o}$ and $\eE_{2,o}$ are 
both $\sigma$-limit stable objects with the same numerical classes, 
the morphism $\widetilde{f}_o$ is an isomorphism. Hence $\widetilde{f}$
 is also an 
isomorphism.  
 \end{proof}
\begin{rmk}\emph{
In~\cite{Tst3}, the author used the result of~\cite[Proposition 3.5.3]{AP}
to show the openness of stability for the case of K3 surfaces. 
In the situation of our paper, the relevant abelian category $\aA^p$
is not noetherian which prevents us to use the result of~\cite{AP}. 
Instead we have used Lemma~\ref{extension}, Lemma~\ref{extension2} to show 
the openness.} 
\end{rmk}
\begin{rmk}\emph{
It seems likely that $\lL_n^{\sigma}(X, \beta)$ is a 
projective variety 
 for a generic choice of $\sigma$, which 
we are unable to prove at this time.
We do not how to construct the moduli space as a GIT quotient. 
Also the main technical 
difficulty to show the properness 
is that we are unable to use extension results of a family 
of objects
as in~\cite[Theorem 4.1.1]{AP}, since $\aA^p$ is not 
noetherian again.}
\end{rmk}

\section{Counting invariants of limit stable objects}\label{section:count}
In this section, we again assume that $X$ is a projective 
Calabi-Yau 3-fold. The purpose of this section is 
to construct virtual counting of $\sigma$-limit 
stable objects, and study their properties. 
\subsection{Definitions of counting invariants}
For $\sigma \in A(X)_{\mathbb{C}}$, $\beta \in H^4(X, \mathbb{Z})$
and $n\in \mathbb{Z}$, let 
$\lL_n^{\sigma}(X, \beta)$
be the algebraic space constructed in Theorem~\ref{ready}.
In this paragraph, we give the definition of
the counting invariant of $\sigma$-limit stable objects 
$L_{n, \beta}(\sigma) \in \mathbb{Z}$ using $\lL_n^{\sigma}(X, \beta)$. 
Since we are unable to conclude that $\lL_n^{\sigma}(X, \beta)$ 
is proper, the integration of virtual classes does not 
make sense. Instead we use K.~Behrend's constructible function~\cite{Beh}
to define counting invariants. 
Recall that Behrend~\cite{Beh} constructs on any scheme $M$,
(more generally $M$ is a Deligne Mumford stack,) 
a canonical constructible function, 
$$\nu_{M}\colon M \lr \mathbb{Z},$$
which depends only on the scheme structure of $M$. 
If $M$ is smooth, $\nu_M$ is given by $\nu_M(p)=(-1)^{\dim M}$. 
Moreover if $M$ is proper and carries a symmetric perfect
obstruction theory, 
one has 
$$\sharp^{\rm{vir}}(M)=\sum_{n\in \mathbb{Z}}ne(\nu_{M}^{-1}(n)), $$
where $\sharp^{\rm{vir}}(M)$ is the integration over the 
virtual cycle, and
$e(\ast)$ is the euler number. 
In our situation, let 
$$\nu_{L}\colon \lL_{n}^{\sigma}(X, \beta) \lr \mathbb{Z},$$
be Behrend's constructible function. 
\begin{defi}\emph{
We define the invariant $L_{n, \beta}(\sigma)\in \mathbb{Z}$
by the formula,} 
$$L_{n, \beta}(\sigma)=\sum_{n\in \mathbb{Z}}ne(\nu_{L}^{-1}(n)).$$
\end{defi}
Note that since an algebraic space of finite type is stratified by 
affine schemes of finite type, its euler number makes sense. 
\begin{rmk}\emph{
Suppose that $\lL_{n}^{\sigma}(X, \beta)$ is a proper algebraic space. 
Then by the same argument as in~\cite[Lemma 2.10]{PT} and \cite{HuTho},
there is a virtual fundamental class, 
$$[\lL_{n}^{\sigma}(X, \beta)]^{\rm{vir}} \in A_0(\lL_n^{\sigma}(X, \beta)).$$
By the above argument, 
 our invariant $L_{n, \beta}(\sigma)$ coincides with the 
integration over the virtual class, 
$$L_{n, \beta}(\sigma)=
\int_{[\lL_{n}^{\sigma}(X, \beta)]^{\rm{vir}}}1.$$}
\end{rmk}
The purpose of this section is to relate the 
invariants $L_{n, \beta}(\sigma)$ to the 
invariants of stable pairs on a Calabi-Yau 3-fold introduced by 
Pandharipande and Thomas~\cite{PT}. 
Let us recall the notion of stable pairs. 
\begin{defi}\emph{\bf{\cite{PT}}}\emph{
A stable pair
consists of data $(F, s)$, 
$$s\colon \oO_X \lr F, $$
where $F$ is a pure one dimensional sheaf and $s$
is a morphism satisfying 
$$\dim \Cok(s)=0.$$}
\end{defi}
Given a stable pair $(F, s)$, we can associate the 
two term complex, 
\begin{align}\label{term}
I^{\bullet}=(\oO_X \stackrel{s}{\lr} F) \in D^b(X),\end{align}
where $F$ is located in degree zero. As we mentioned
in Remark~\ref{quasi-iso}, if $F$ satisfies 
\begin{align}\label{satisfy}
\ch_2(F)=\beta, \quad \ch_3(F)=n,
\end{align}
we have $I^{\bullet} \in L_n(X, \beta)$. 
By abuse of notation, we also call the two term complexes
(\ref{term}) as stable pairs. 
In~\cite{PT}, the moduli space of stable pairs $(F, s)$
satisfying the condition (\ref{satisfy}) is constructed
as a projective variety, 
and denoted by $P_n(X, \beta)$. The obstruction theory 
on $P_n(X, \beta)$ is obtained from the deformation theory of the 
 two term complexes 
$I^{\bullet}=(\oO_X \stackrel{s}{\to}F)$. 
\begin{defi}\emph{\bf{\cite{PT}}}\label{PTinv}\emph{
A PT-invariant $P_{n, \beta} \in \mathbb{Z}$ is defined by 
$$P_{n, \beta}=\int_{[P_n(X, \beta)]^{\rm{vir}}}1
= \sum_{n\in \mathbb{Z}}ne(\nu_{P}^{-1}(n))\in \mathbb{Z}.$$
Here $\nu_{P}\colon P_n(X, \beta)\to \mathbb{Z}$ is
Behrend's constructible function.}
\end{defi}

\subsection{Limit stable objects and stable pairs}
The purpose of this paragraph is 
to investigate the relationship between 
$L_{n, \beta}(\sigma)$ and $P_{n, \beta}$. 
First we show the following lemma. 
\begin{lem}\label{stabpair}
An object $E\in L_n(X, \beta)$ 
(cf.~(\ref{L1}))
is isomorphic to 
a stable pair (\ref{term}) if and only if $\hH^0(E)$ is zero 
dimensional.
\end{lem}
\begin{proof}
Only if part is obvious, so we show the if part. 
For an object $E\in L_n(X, \beta)$, suppose that $\hH^0(E)$
is zero dimensional. 
Applying $\Hom(\ast, \oO_X[1])$ for the triangle
$\hH^{-1}(E)[1] \to E \to \hH^0(E)$, we obtain 
the exact sequence, 
$$\Hom(\hH^0(E), \oO_X[1]) \to \Hom(E, \oO_X[1])
\to \Hom(\hH^{-1}(E), \oO_X) \to \Hom(\hH^0(E), \oO_X[2]).$$
Since $\hH^0(E)$ is zero dimensional, the Serre duality implies
$$\Hom(\hH^0(E), \oO_X[j])=H^{3-j}(X, \hH^0(E))=0,$$
for $j=1, 2$. 
Hence we have the isomorphism 
$$\Hom(E, \oO_X[1]) \stackrel{\cong}{\lr}
 \Hom(\hH^{-1}(E), \oO_X).$$
  By Lemma~\ref{lem:C}, we have $\hH^{-1}(E)=I_C$
 for a one dimensional subscheme $C\subset X$. 
 Therefore there is a morphism $u\colon E\to \oO_X[1]$ corresponding 
 to the inclusion $I_C \subset \oO_X$. Let us take the 
 distinguished triangle, 
 $$ \oO_X \lr F \lr E \stackrel{u}{\lr} \oO_X[1].$$
 It is enough to show that $F$ is a pure one dimensional sheaf. 
 Since $E\in L_n(X, \beta)$, it is obvious that 
 $F \in \Coh_{\le 1}(X)$. For a closed point $x\in X$, we have 
 the exact sequence, 
 $$0=\Hom(\oO_x, \oO_X) \lr \Hom(\oO_x, F) \lr \Hom(\oO_x, E).$$
 Since $E\in \aA_{1/2}^p$, we have $\Hom(\oO_x, E)=0$, 
 hence $\Hom(\oO_x, F)=0$. Therefore $F$ is a pure sheaf. 
 \end{proof}
  In the following, we focus on $\sigma\in A(X)_{\mathbb{C}}$ 
of the form, 
\begin{align}\label{form}
\sigma=k\omega +i\omega, \quad k\in \mathbb{R},\end{align}
and see how $\lL_n^{\sigma}(X, \beta)$ and $L_{n, \beta}(\sigma)$
vary under change of $k\in \mathbb{R}$. 
The advantage of setting $\sigma$ as (\ref{form}) is 
as follows. 
\begin{lem}\label{adv}
For $F\in \Coh_{\le 1}(X)$ and $E\in L_n(X, \beta)$, 
the condition $\phi_{\sigma}(F)\preceq \phi_{\sigma}(E)$ 
(resp. $\phi_{\sigma}(F)\succeq \phi_{\sigma}(E)$)
implies
\begin{align}\label{adv2}
k\le -\frac{\mu_{i\omega}(F)}{2}, \quad 
\left( \mbox{resp. }k\ge -\frac{\mu_{i\omega}(F)}{2}.\right)
\end{align}
\end{lem}
\begin{proof}
If $\sigma=k\omega +i\omega$, we have
$$\mu_{\sigma}(F)
=\mu_{i\omega}(F)-k, 
\quad -\frac{3B\omega^2}{\omega^3}=-3k,
$$
for $B=k\omega$. Hence the
inequality (\ref{adv2}) follows from the same 
argument as in Lemma~\ref{equ}. 
\end{proof}
We also note that $(k\omega, \omega)$-twisted 
(semi)stable sheaves coincide with $\omega$-Gieseker
(semi)stable sheaves. 

We set $\mu_{n, \beta} \in \mathbb{Q}$ as follows, 
 $$\mu_{n, \beta}=\max \left\{ 
 \frac{n-m(\beta'')}{\omega \beta'} :  
 0\neq \beta', \beta'' \in \nN(\beta)
  \mbox{ and }\beta=\beta'+\beta''
 \right\}.$$  
 Since $\nN(\beta)$ is a finite set, we have $\mu_{n, \beta}<\infty$. 
The following is the main result of this section. 
\begin{thm}\label{LiPT}
Let $\sigma =k\omega +i\omega$ for $k\in \mathbb{R}$. We have, 
\begin{align*}
\lL_n^{\sigma}(X, \beta) &=P_{n}(X, \beta), \quad 
 L_{n, \beta}(\sigma) =P_{n, \beta}, \quad 
\mbox{ if }k<-\mu_{n, \beta}/2, \\
\lL_n^{\sigma}(X, \beta) &=P_{-n}(X, \beta), \quad 
 L_{n, \beta}(\sigma) =P_{-n, \beta}, \quad 
\mbox{ if }k>\mu_{-n, \beta}/2.\end{align*}
\end{thm}
\begin{proof}
First assume $k<-\mu_{n, \beta}/2$ and take $E\in L_n^{\sigma}(X, \beta)$.
(cf.~(\ref{L2}).)
In order to show $E$ is isomorphic to a stable pair (\ref{term}), it 
suffices to check that $\hH^0(E)$ is zero dimensional by Lemma~\ref{stabpair}. 
Suppose by a contradiction that $\hH^0(E)$ is one dimensional, 
and set 
$$\beta'=\ch_2(\hH^0(E)) \neq 0, \quad \beta''=-\ch_2 (\hH^{-1}(E)).$$ 
By Lemma~\ref{lem:final}, we have 
$$\ch_3(\hH^0(E)) \le n-m(\beta'').$$
Since $E$ is $\sigma$-limit stable, we must have
$\phi_{\sigma}(E)\prec \phi_{\sigma}(\hH^0(E))$. 
Lemma~\ref{adv} implies that 
$$k\ge -\frac{\mu_{i\omega}(\hH^0(E))}{2}=
-\frac{\ch_3(\hH^0(E))}{2\omega \beta'} \ge 
-\frac{n-m(\beta'')}{2\omega \beta'}\ge -\frac{\mu_{n, \beta}}{2}.$$
This contradicts that $k<-\mu_{n, \beta}/2$, hence 
$E$ is isomorphic to a stable pair. 

Conversely take a stable pair $E\cong 
(\oO_X \stackrel{s}{\to} F) \in P_n(X, \beta)$.
We have to check the conditions $(a)$, $(b)$ in Lemma~\ref{prop:char}. 
Let $E\twoheadrightarrow G$ be a strict epimorphism
 as in $(a)$ in Lemma~\ref{prop:char}. 
By Remark~\ref{rmk:eq}, 
$G$ is obtained as a surjection of sheaves $\hH^0(E) \twoheadrightarrow G$. 
Since $\hH^0(E)$ is zero dimensional, $G$ is also a zero dimensional 
sheaf, hence $G\notin \aA_{1/2}^p$ provided $G\neq 0$. This means that $(a)$ does not
occur, so it is enough 
to check $(b)$ in Lemma~\ref{prop:char}. 

Let $F' \hookrightarrow E$ be a strict monomorphism 
 as in $(b)$ in Lemma~\ref{prop:char}, 
and set $\beta'=\ch_2(F')$. 
Let $C\subset X$ be a one dimensional subscheme with 
$\hH^{-1}(E)=I_C$. We can take subsheaves
$F_1 \subset \oO_C$, $F_2 \subset \hH^0(E)$
as in Lemma~\ref{asfor}, such that $F'$ is written as an 
extension 
$$0 \lr F_1 \lr F' \lr F_2 \lr 0.$$
Since $\oO_C/F_1 \cong \oO_{C'}$ for a subscheme $C' \subset C$, 
we have 
\begin{align*}
\ch_3(F_1) &=\ch_3(\oO_C)-\ch_3(\oO_{C'}) \\
&\le \ch_3(\oO_C)-m(\beta''),
\end{align*}
where $\beta''=\ch_2(\oO_{C'})$. Hence we have
\begin{align*}
\ch_3(F') &= \ch_3(F_1)+\ch_3(F_2) \\
&\le \ch_3(\oO_C) +\ch_3(\hH^0(E)) -m(\beta'') \\
&=n-m(\beta'').
\end{align*}
We obtain the inequality,
\begin{align*}
k<-\frac{\mu_{n, \beta}}{2}\le -\frac{n-m(\beta'')}{2\beta' \omega}
\le -\frac{\mu_{i\omega}(F')}{2},
\end{align*}
which implies $\phi_{\sigma}(F')\prec \phi_{\sigma}(E)$
by Lemma~\ref{adv}.
Therefore $(b)$ in Lemma~\ref{prop:char} holds, hence 
$E\in L_n^{\sigma}(X, \beta)$. 

By the above arguments, 
the set of $\mathbb{C}$-valued points of 
$\lL_n^{\sigma}(X, \beta)$ and $P_n(X, \beta)$ 
are identified. Therefore we have the isomorphism of the 
moduli spaces, 
$$\lL_n^{\sigma}(X, \beta)\cong P_n(X, \beta),$$
since both are open algebraic subspaces of $\mM_0^{\rm{et}}$. 
In particular $L_{n, \beta}(\sigma)=P_{n, \beta}$ follows. 
When $k>\mu_{-n, \beta}/2$,
we have
$$L_{n, \beta}(\sigma)=L_{-n, \beta}(\sigma^{\vee})=P_{-n, \beta},$$
by applying the dualizing functor $\mathbb{D}$
and using Lemma~\ref{compati}. 
\end{proof}

\begin{rmk}\emph{ 
The wall-crossing phenomena for stable pairs is also studied in 
Bayer's polynomial stability conditions~\cite[Paragraph 6.2]{Bay}. However our
wall-crossing is crucially different from Bayer's wall-crossing. 
In fact the complexified
ample cone $A(X)_{\mathbb{C}}$, the stability parameter 
in our stability conditions, behaves itself as a wall in the wall-crossing 
of Bayer~\cite[Paragraph 6.2]{Bay}.}
\end{rmk}

\subsection{Wall-crossing phenomena of limit stable objects}\label{sub:wall}
Let $\sigma=k\omega+i\omega$ be as in the previous paragraph. 
In this paragraph, we investigate how $\sigma$-limit stable 
objects vary under change of $k\in \mathbb{R}$. 
As we have seen in Theorem~\ref{LiPT},  
limit stable objects coincide with stable pairs for $k\ll 0$, 
and the dual of stable pairs for $k\gg 0$. 
We look at the wall-crossing phenomena more closely, which 
hopefully might  
be helpful for the rationality conjecture 
of the generating functions of PT-invariants, proposed in~\cite{PT}.

For an effective class $\beta \in H^4(X, \mathbb{Z})$, 
we set $\sS(\beta) \subset \mathbb{R}$ as 
$$\sS(\beta)\cneq \left\{
\frac{m}{2\omega \gamma} : 0\neq \gamma \in \nN(\beta),  m\in \mathbb{Z}
\right\} \subset \mathbb{R},$$
where $\nN(\beta)$ is introduced in (\ref{N(beta)}). 
Note that $\sS(\beta)$ is a discrete subset in $\mathbb{R}$
because $\nN(\beta)$ is a finite set. 
In the following, we see that $\sS(\beta)$ behaves as the 
set of walls. 
\begin{prop}
Let $\cC \subset \mathbb{R}\setminus
\sS(\beta)$ be one of the connected components. 
For $k, k' \in \cC$, we have 
$$\lL_n^{\sigma}(X, \beta)=\lL_n^{\sigma'}(X, \beta), $$
where $\sigma =k\omega +i\omega$, $\sigma'=k'\omega +i\omega$. 
In particular the function, 
$$\mathbb{R} \ni k \longmapsto L_{n, \beta}(k\omega +i\omega)\in\mathbb{Z}, $$
is constant on $\cC$. 
\end{prop}
\begin{proof}
For $E\in L_n^{\sigma}(X, \beta)$, assume that 
$E\notin L_n^{\sigma'}(X, \beta)$. 
Then at least one of the conditions $(a)$ or $(b)$
in Lemma~\ref{lem:C} does not hold. 
Suppose that $(a)$ does not hold and let $E\twoheadrightarrow G$
be a strict epimorphism with $G\in \Coh_{\le 1}(X)$, 
which destablizes $E$.
Then Lemma~\ref{equ} and Lemma~\ref{adv} show, 
\begin{align}\label{RHS}
k' \le -\frac{\mu_{i\omega}(G)}{2}.\end{align}
Since $\ch_2(G) \in \nN(\beta)$ by
 Remark~\ref{rmk:eq}, the right hand side of 
(\ref{RHS}) is an element of $\sS(\beta)$. 
Therefore $k$ satisfies the inequality
$k< -\mu_{i\omega}(G)/2$, 
which implies $\phi_{\sigma}(E)\succ \phi_{\sigma}(G)$. 
This contradicts that $E\in L_n^{\sigma}(X, \beta)$. 

The case that the condition 
$(b)$ in Lemma~\ref{lem:C} does not hold 
is similarly discussed, noting Lemma~\ref{asfor} which shows that
$\ch_2(F) \in \nN(\beta)$ for destablizing monomorphism $F\hookrightarrow E$.
\end{proof}

Next we investigate 
the wall-crossing phenomena at some point $-\mu/2 \in \sS(\beta)$. 
Let $\cC_{-}$, $\cC_{+}$ be connected components in $\mathbb{R}\setminus 
\sS(\beta)$ such that
$$\cC_{-}\subset \mathbb{R}_{<-\mu/2}, \quad
 \cC_{+}\subset \mathbb{R}_{>-\mu/2},
\quad \overline{\cC}_{-}\cap \overline{\cC}_{+}
=\left\{-\frac{\mu}{2}\right\}.$$
Let us take $k_{-}\in \cC_{-}$, $k_{+}\in \cC_{+}$
and $k_0=-\mu/2$. 
We set $\sigma_{\ast}=k_{\ast}\omega+i\omega$
for $\ast=\pm, 0$. 
We have the following proposition. 
\begin{prop}\label{wall-cross}
(i) Assume that $E\in L_n^{\sigma_{-}}(X, \beta)$ is
not $\sigma_{+}$-limit stable. Then there is 
an exact sequence in $\aA^p_{1/2}$, 
\begin{align}\label{wall1}
0 \lr F \lr E \lr G \lr 0,
\end{align}
such that $F$ is $\omega$-Gieseker semistable sheaf with 
$\mu_{i\omega}=\mu$ and $G$ is $\sigma_{+}$-limit stable
with $\phi_{\sigma_{+}}(F) \succ \phi_{\sigma_{+}}(G)$. 
i.e. (\ref{wall1}) is a Harder-Narasimhan filtration in $\sigma_{+}$. 
The object $G$ is also $\sigma_0$-limit stable. 

(ii) Assume that $E\in L_n^{\sigma_{+}}(X, \beta)$ is
not $\sigma_{-}$-limit stable. Then there is 
an exact sequence in $\aA^p_{1/2}$, 
\begin{align}\label{wall2}
0 \lr G \lr E \lr F \lr 0,
\end{align}
such that $F$ is $\omega$-Gieseker semistable sheaf with 
$\mu_{i\omega}=\mu$ and $G$ is $\sigma_{-}$-limit stable
with $\phi_{\sigma_{-}}(G) \succ \phi_{\sigma_{-}}(F)$. 
i.e. (\ref{wall2}) is a Harder-Narasimhan filtration in $\sigma_{-}$. 
The object $G$ is also $\sigma_0$-limit stable.
\end{prop}
\begin{proof}
The proof of (ii) is identical to (i), so we 
only show (i). 
Assume that $E\in L_n^{\sigma_{-}}(X, \beta)$ is not 
$\sigma_{+}$-limit stable. By Lemma~\ref{prop:char}, we have 
one of the two possibilities. 

$(a')$ There is a non-zero pure one dimensional sheaf $G$
which admits a strict epimorphism $E\twoheadrightarrow G$
in $\aA_{1/2}^p$ with $\phi_{\sigma_{+}}(E)\succeq \phi_{\sigma_{+}}(G)$.

$(b')$ There is a non-zero pure one dimensional sheaf $F$
which admits a strict monomorphism $F\hookrightarrow E$
in $\aA_{1/2}^p$ with $\phi_{\sigma_{+}}(F)\succeq \phi_{\sigma_{+}}(E)$. 

Suppose that $(a')$ occurs. By Lemma~\ref{adv}, we see that 
$$k_{-}< k_{+} \le -\frac{\mu_{i\omega}(G)}{2}.$$
Therefore we also have $\phi_{\sigma_{-}}(E)\succ \phi_{\sigma_{-}}(G)$, 
which contradicts that $E\in L_n^{\sigma_{-}}(X, \beta)$. 
Hence $(b')$ occurs. 
Let $F\hookrightarrow E$ be as in $(b')$. 
Since $\aA_{1/2}^p$ is of finite length, one can 
take such $F$ to be maximal, i.e. there is no non-trivial
strict monomorphism $F \hookrightarrow \tilde{F} \hookrightarrow E$
such that $\tilde{F}$ is also pure one dimensional sheaf 
with $\phi_{\sigma_{+}}(\tilde{F}) \succeq \phi_{\sigma_{+}}(E)$. 
Since 
$\phi_{\sigma_{+}}(F)\succeq \phi_{\sigma_{+}}(E)$ and 
$\phi_{\sigma_{-}}(F)\prec \phi_{\sigma_{-}}(E)$, 
we have 
$$k_{-} < -\frac{\mu_{i\omega}(F)}{2} \le k_{+}.$$
Since $\ch_2(F) \in \nN(\beta)$, we have 
$-\mu_{i\omega}(F)/2 \in \sS(\beta)$. 
Therefore we have
$$\mu_{i\omega}(F)=\mu, \quad 
k_{+}>-\frac{\mu_{i\omega}(F)}{2},$$
which implies $\phi_{\sigma_{+}}(F)\succ \phi_{\sigma_{+}}(E)$. 
In order to show $F$ is a $\omega$-Gieseker semistable
sheaf, it is enough to show that $F$ is $\sigma_{+}$-limit semistable. 
Let $F' \subset F$ be a strict monomorphism in 
$\aA_{1/2}^p$. 
Note that $F'$ is also pure one dimensional sheaf. 
If $\phi_{\sigma_{+}}(F') \succeq \phi_{\sigma_{+}}(F)$, 
then $\phi_{\sigma_{+}}(F') \succeq \phi_{\sigma_{+}}(E)$, 
hence $\mu_{i\omega}(F')=\mu$ by the same argument as above. 
It follows that $F$ is $\sigma_{+}$-limit semistable. 

Let us take the exact sequence in $\aA_{1/2}^p$, 
$$0 \lr F \lr E \lr G \lr 0.$$
We want to show that $G$ is $\sigma_{*}$-limit stable
for $\ast =+, 0$. 
We show the case of $\ast =0$, as the proof for the other case is similar. 
Suppose the contrary. 
By the same argument as above, there is a strict 
monomorphism $F'' \hookrightarrow G$ in $\aA_{1/2}^p$ 
such that $\phi_{\sigma_{0}}(F'') \succeq \phi_{\sigma_{0}}(G)$. 
We have 
$$k_{+}>k_{0}\ge -\frac{\mu_{i\omega}(F'')}{2},$$
by Lemma~\ref{adv}, which in turn implies 
$\phi_{\sigma_{+}}(F'')\succ \phi_{\sigma_{+}}(E)$. 
Let $F'''$ be the kernel of the composition of the strict epimorphisms 
$$E \twoheadrightarrow G \twoheadrightarrow G/F''.$$ 
Then we have the non-trivial strict monomorphism, 
$F \hookrightarrow F'' \hookrightarrow E$ with 
$\phi_{\sigma_{+}}(F'') \succ \phi_{\sigma_{+}}(E)$, 
which is a contradiction since $F \hookrightarrow E$ is maximal. 
\end{proof}

Let $Z^{\rm{PT}}_{\beta}$ be the generating series, 
\begin{align}\label{generating}
Z^{\rm{PT}}_{\beta}(q)=\sum _{n\in \mathbb{Z}}P_{n, \beta}q^n
\in \mathbb{Q}\db[q\db]. 
\end{align}
In~\cite[Conjecture 3.2]{PT}, Pandharipande and Thomas conjecture that 
the generating series
(\ref{generating}) is a rational function of $q$, invariant 
under $q \mapsto 1/q$. This conjecture (rationality conjecture) is 
solved when $\beta$ is an irreducible curve class 
case in~\cite{PT3}, and
the crucial point is to find a relationship between 
$P_{n, \beta}$ and $P_{-n, \beta}$. 
By Theorem~\ref{LiPT}, it is possible to 
obtain such a relationship in a general situation by
establishing   
a wall-crossing formula of our invariants $L_{n, \beta}(\sigma)$. 
Suppose for instance that any $F\in \Coh_{\le 1}(X)$ which 
appears in the sequence (\ref{wall1}), (\ref{wall2})
 is in fact $\omega$-stable, and satisfies
\begin{align}\label{Gst}
(\ch_2(F), \ch_3(F))=(\beta', n').
\end{align}
Let
$\lL^{-} \subset \lL^{\sigma_{-}}_n(X, \beta)$ 
be the unstable locus in $\sigma_{+}$-limit stability, 
$\lL^{+} \subset \lL^{\sigma_{+}}_n(X, \beta)$
the similar locus, and $M_{n'}(X, \beta')$
 the moduli space of $\omega$-Gieseker
stable sheaf $F\in \Coh_{\le 1}(X)$ 
satisfying (\ref{Gst}).
The destabilizing sequences (\ref{wall1}), (\ref{wall2}) yield
the following diagram,
\begin{align*}\xymatrix{
 \lL^{-} \ar[dr]^{\pi_{-}} & & 
\lL^{+}  \ar[dl]_{\pi_{+}} \\
& M_{n'}(X, \beta')\times \lL^{\sigma_{0}}_{n''}(X, \beta'').&}
\end{align*}
Here $\beta'+\beta''=\beta$ and $n'+n''=n$. 
From the above diagram, one might expect the formula
something like
\begin{align}\label{expect}
L_{n, \beta}(\sigma_{-})-L_{n, \beta}(\sigma_{+})=
(\sharp \Ext^1(G, F)-\sharp \Ext^1(F, G))
N_{n', \beta'}L_{n'', \beta''}(\sigma_0),
\end{align}
where $N_{n', \beta'}$ is the virtual counting, 
$$N_{n', \beta'}=\int_{[M_{n'}(X, \beta')]^{\rm{vir}}}1 \in \mathbb{Z}.$$
We expect that the contribution of the term
 $(\sharp \Ext^1(G, F)- \sharp \Ext^1(F, G))$
is given by 
$$(-1)^{n'-1}\chi(F, G)=(-1)^{n'-1}n',$$
so the 
right hand side of (\ref{expect}) should be $(-1)^{n'-1}n'
N_{n', \beta'}L_{n'', \beta''}(\sigma_0)$. 

In case a destabilizing sheaf $F\in \Coh_{\le 1}(X)$ 
which appears in (\ref{wall1}) or (\ref{wall2}) is strictly semistable, 
the construction of the invariant $N_{n', \beta'}$ is problematic.
It seems that Joyce's motivic invariants of moduli 
stacks~\cite[Definition 3.18]{Joy4} are relevant for this problem, 
although Joyce's invariants are not 
deformation invariant as they do not involve virtual classes.
Hopefully  
it is possible to involve virtual classes 
(probably 
using Behrend's constructible function~\cite{Beh}, )
and the following wall-crossing formula should hold.
\begin{conj}\label{conj}
There is a virtual counting of 
$\omega$-Gieseker semistable sheaves $F\in \Coh_{\le 1}(X)$
with $(\ch_2(F), \ch_3(F))=(\beta', n')$, denoted by 
$N_{n', \beta'}\in \mathbb{Q}$, such that 
\begin{align}\label{expect2}L_{n, \beta}(\sigma_{-})-L_{n, \beta}(\sigma_{+})=
\sum (-1)^{n'-1}n'N_{n', \beta'}L_{n'', \beta''}(\sigma_0).\end{align}
Here in the above sum, $(\beta', n')$, $(\beta'', n'')$ must
satisfy 
$\beta'+\beta''=\beta$, 
 $n'+n''=n$ and $n'/\omega\beta'=\mu$. 
\end{conj}
Note that a term in the right hand side of (\ref{expect2})
is non-zero only if $\beta', \beta'' \in \nN(\beta)$, so
there are only finite number of non-zero terms. 
Also for a non-zero term in (\ref{expect2}), 
the class $\beta''$ is smaller than $\beta$, i.e. 
$0\le \beta'' \cdot H < \beta \cdot H$ for an ample divisor $H$. 
Hence
if Conjecture~\ref{conj} is true, we can describe how the invariants
 $L_{n, \beta}(\sigma)$ vary under change of $\sigma$
 inductively on $\beta$, and eventually
provides a relationship between $P_{n, \beta}$ and $P_{-n, \beta}$. 
We expect this relationship will show the 
rationality conjecture of the generating series (\ref{generating}). 

In the next paper~\cite{Togen}, we will proceed this idea further using 
D.~Joyce's work~\cite{Joy4}
on the wall-crossing formula of
counting invariants of semistable objects in abelian categories. 
It will turn out in~\cite{Togen}
 that the similar rationality property holds for 
the generating functions of euler numbers of the moduli spaces of stable 
pairs, using the results in this paper. 
We remark that Joyce's work is applied for the invariants 
without virtual fundamental cycles. However
by the recent progress in this field~\cite{K-S}, \cite{J-S1}, 
\cite{J-S2}, we guess that the similar wall-crossing formula 
should hold after involving virtual classes. 
At this moment, the works~\cite{K-S}, \cite{J-S2}
are not enough to conclude the rationality conjecture. 
(The work~\cite{K-S} assumes~\cite[Conjecture~4]{K-S} 
to show the main result~\cite[Theorem~8]{K-S},
and the result of~\cite{J-S2} is only applied 
for counting invariants of coherent sheaves, not 
for those of objects in the derived category.)
Finally we mention that T.~Bridgeland~\cite{Brs5}
proved the rationality conjecture assuming the 
main result of Kontsevich-Soibelman~\cite[Theorem~8]{K-S}, 
without using any notion of stability conditions.

\section{Examples}\label{section:Ex}
In this section, we see the wall-crossing phenomena 
of limit stable objects in several examples. 

\subsection{$\beta$ is an irreducible curve class}
Suppose that $\beta\in H^4(X, \mathbb{Z})$ is an irreducible class, i.e. 
$\beta=[C]$ for an irreducible and reduced curve $C\subset X$. 
For $\sigma=k\omega +i\omega$, we have 
\begin{align*}
\lL_n^{\sigma}(X, \beta)= \left\{ \begin{array}{ll} 
P_n(X, \beta) & \mbox{ if }k<-\frac{n}{2\omega \beta}, \\
P_{-n}(X, \beta)& \mbox{ if }k> -\frac{n}{2\omega \beta}.
\end{array} \right.
\end{align*}
Note that $\mu_{n, \beta}=n/\omega \beta$ in this case, 
so
Theorem~\ref{LiPT}
yields the above result. 
The formula (\ref{expect2}) becomes
$$P_{n, \beta}-P_{-n, \beta}=(-1)^{n-1}n N_{n, \beta},$$
which is proved in~\cite[Proposition 2.2]{PT3}. 
Hence Conjecture~\ref{conj} is true in this case.

\subsection{$\beta$ is a reducible curve class}
Suppose that there are smooth rational curves $C_1$, $C_2$
on $X$ such that 
$$N_{C_i/X}\cong \oO_{C_i}(-1)^{\oplus 2}, \quad 
\beta=[C_1]+[C_2], \quad C_1 \cap C_2=\{p\},$$
where $C_1 \cap C_2$ is the scheme theoretic intersection. 
Let $C=C_1 \cup C_2$, 
 $d_i =\omega \cdot C_i$ and assume that $d_1>d_2>0$. 
This is possible if $C_1$ and $C_2$ determine linearly independent
homology classes in $H_2(X, \mathbb{R})$. 
As for $\lL_1^{\sigma}(X, \beta)$, we have the following, 
\begin{align*}
\lL_1^{\sigma}(X, \beta)= \left\{ \begin{array}{ll} 
P_1(X, \beta)\cong \Spec \mathbb{C} & \mbox{ if }k<-\frac{1}{2(d_1+d_2)}, \\
P_{-1}(X, \beta)=\emptyset & \mbox{ if }k> -\frac{1}{2(d_1+d_2)}.
\end{array} \right.
\end{align*}
In fact $\mu_{1, \beta}=1/(d_1+d_2)$ and $\mu_{-1, \beta}=-1/(d_1+d_2)$
in this case, so we can apply Theorem~\ref{LiPT}.  
The set of 
stable pairs $E\in \aA^p$ with $(\ch_2(E), \ch_3(E))=(1, \beta)$
consists of one element $\{I_C[1]\}$, and 
we can easily compute 
$$\Ext_X^1(I_C[1], I_C[1])=0. $$  
Hence scheme theoretically $P_1(X, \beta)$
is isomorphic to $\Spec \mathbb{C}$.
If $k>-1/2(d_1+d_2)$, then the exact sequence in $\aA^p$, 
$$0 \lr \oO_C \lr I_C[1] \lr \oO_X[1] \lr 0, $$
destabilizes $I_C[1]$. According to Proposition~\ref{wall-cross}, 
 we might obtain stable object 
as an extension, 
$$0 \lr \oO_X[1] \lr E \lr \oO_C \lr 0.$$
However since $\Ext_X^1(\oO_C, \oO_X[1])=H^1(\oO_C)=0$, 
the above sequence splits, so $E$ is not $\sigma$-limit 
stable.
In fact we can check that $P_{-1}(X, \beta)$ is empty 
in this case. 
The counting invariants are as follows, 
\begin{align*}
L_{1, \beta}(\sigma)= \left\{ \begin{array}{cl} 
1 & \mbox{ if }k<-\frac{1}{2(d_1+d_2)}, \\
0 & \mbox{ if }k> -\frac{1}{2(d_1+d_2)}.
\end{array} \right.
\end{align*}
The formula (\ref{expect2}) is easily checked to 
hold in this case. 

Next let us investigate $\lL_2^{\sigma}(X, \beta)$. 
The result is as follows.
\begin{align*}
\lL_2^{\sigma}(X, \beta)= \left\{ \begin{array}{cl} 
P_2(X, \beta)\cong C  & \mbox{ if }k<-\frac{1}{2d_2}, \\
\mathbb{P}^1 & \mbox{ if } -\frac{1}{2d_2}<k< -\frac{1}{d_1+d_2}, \\
P_{-2}(X, \beta)=\emptyset & \mbox{ if }k>-\frac{1}{d_1+d_2}.
\end{array} \right.
\end{align*}
In this case, we have $\mu_{2, \beta}=1/d_2$ and 
$\mu_{-2, \beta}=-2/(d_1+d_2)$. 
Also giving a point of 
$P_2(X, \beta)$ is equivalent to choosing a closed point
of $C$.  
By Theorem~\ref{LiPT} together with some more arguments, 
we see
$$\lL_2^{\sigma}(X, \beta)=P_2(X, \beta)\cong C, $$ 
for $k<-1/2d_2$. 

Suppose that 
$$-\frac{1}{2d_2}<k< -\frac{1}{d_1+d_2}.$$ 
Any stable pair $E\in P_2(X, \beta)$ admits an exact sequence in $\aA^p$, 
\begin{align}\label{more}
0 \lr I_C[1] \lr E \lr \oO_x \lr 0,\end{align}
for $x\in C$. If $x\in C_2$, then we have the exact sequence 
in $\aA^p$, 
$$ 0\lr \oO_{C_2} \lr E \lr I_{C_1}[1] \lr 0, $$
which destabilizes $E$. 
In fact $E\in \aA^p$ given in (\ref{more})
is $\sigma$-limit stable if and only if $x\notin C_2$. 
Hence $C_1 \setminus \{p\}=\mathbb{C}$ is embedded into 
$\lL_2^{\sigma}(X, \beta)$. It is compactified by
adding a point corresponding to the (unique) extension, 
\begin{align}\label{more2}
0 \lr I_{C_1}[1] \lr E' \lr \oO_{C_2} \lr 0.\end{align}
One can check that 
$$L_n^{\sigma}(X, \beta)=(C_1 \setminus \{p\}) \cup \{E'\}.$$
These objects
are also obtained as two term complexes
$\oO_X \stackrel{s}{\to} F$,
where $F \in \Coh_{\le 1}(X)$ is a unique non-trivial extension, 
\begin{align}\label{uni}
0 \lr \oO_{C_1} \lr F \lr \oO_{C_2} \lr 0.\end{align}
From these observations, we see that 
$$\lL_2^{\sigma}(X, \beta)=\mathbb{P}(H^0(F)) \cong \mathbb{P}^1.$$
Finally suppose that $k>-1/(d_1+d_2)$. 
Then for a two term complex $E=(\oO_X \stackrel{s}{\to} F)$
where $F$ is as in (\ref{uni}), the exact sequence in $\aA^p$, 
$$0 \lr F \lr E \lr \oO_X[1] \lr 0, $$
destabilizes $E$. Also there are no non-trivial extensions, 
$$0 \lr \oO_X[1] \lr E' \lr F \lr 0, $$
because $\Hom(F, \oO_X[2])=H^1(F)=0$. 
In this case, one can check that $\lL_2^{\sigma}(X, \beta)=
P_{-2}(X, \beta)=\emptyset$. 
For the counting invariants, we obtain
\begin{align*}
L_{2, \beta}(\sigma)= \left\{ \begin{array}{cl} 
-1  & \mbox{ if }k<-\frac{1}{2d_2}, \\
-2 & \mbox{ if } -\frac{1}{2d_2}<k< -\frac{1}{d_1+d_2}, \\
0 & \mbox{ if }k>-\frac{1}{d_1+d_2}.
\end{array} \right.
\end{align*}
The formula (\ref{expect2}) also holds in this case. 
For instance, take 
$$k_{-}<-\frac{1}{2d_2}, \quad k_0=-\frac{1}{2d_2}, \quad 
-\frac{1}{2d_2}<k_{+}< -\frac{1}{d_1+d_2}.$$
A term in the sum of (\ref{expect2}) is non-zero only if
$\beta'=[C_2]$, $n'=1$, $\beta''=[C_1]$ and $n''=1$.
We have $N_{n', \beta'}=1$, and $L_{n'', \beta''}(\sigma_0)=1$. 
We can check (\ref{expect2}) as follows,
\begin{align*}
&L_{2, \beta}(\sigma_{-})-L_{2, \beta}(\sigma_{+})=-1-(-2)=1, \\
&(-1)^{1-1}n' N_{1, [C_2]}L_{1, [C_1]}(\sigma_0)=(-1)^{1-1}1\cdot 
1=1.
\end{align*}
\subsection{$\beta$ is a multiple curve class}
Let $C\subset X$ be a smooth rational curve with 
$$N_{C/X}\cong \oO_C(-1)^{\oplus 2}, \quad \beta=2[C].$$
Let $\sigma=k\omega +i\omega$ and set $d=\omega \cdot C$. 
For $\lL_3^{\sigma}(X, \beta)$, we have the following, 
\begin{align*}
\lL_3^{\sigma}(X, \beta)= \left\{ \begin{array}{ll} 
P_3(X, \beta)\cong \mathbb{P}^1  & \mbox{ if }k<-\frac{1}{d}, \\
P_{-3}(X, \beta) = \emptyset & \mbox{ if }k>-\frac{1}{d}.
\end{array} \right.
\end{align*}
In this case, we have $\mu_{3, \beta}=2/d$
and $\mu_{-3, \beta}=-2/d$
so Theorem~\ref{LiPT} is applied. We have 
$P_3(X, \beta) \cong \mathbb{P}^1$ 
for $k<-1/d$ by~\cite[Section 4]{PT}. 
If $k>-1/d$, we have the exact 
sequence in $\aA^p$ for $E\in P_3(X, \beta)$, 
$$0 \lr \oO_C(1) \lr E \lr I_C[1] \lr 0, $$
which destabilizes $E$. Since $\Hom(\oO_C(1), I_C[2])=0$, 
there is no non-trivial extension, 
$$0 \lr I_C[1] \lr E' \lr \oO_C(1) \lr 0, $$
and in fact $\lL_3^{\sigma}(X, \beta)=P_{-3}(X, \beta)=\emptyset$
in this case. The counting invariants are 
\begin{align*}
L_{3, \beta}(\sigma)= \left\{ \begin{array}{cl} 
-2  & \mbox{ if }k<-\frac{1}{d}, \\
0 & \mbox{ if }k>-\frac{1}{d}.
\end{array} \right.
\end{align*}
The formula (\ref{expect2}) also holds in this case. 

In the same way, 
$\lL_4^{\sigma}(X, \beta)$ is as follows, 
\begin{align*}
\lL_4^{\sigma}(X, \beta)= \left\{ \begin{array}{cl} 
P_4(X, \beta)  & \mbox{ if }k<-\frac{3}{2d}, \\
\Spec \mathbb{C} & \mbox{ if }-\frac{3}{2d}<k < -\frac{1}{d}, \\
P_{-4}(X, \beta)=\emptyset & \mbox{ if }k>-\frac{1}{d}.
\end{array} \right.
\end{align*}
Note that $\mu_{4, \beta}=3/d$ and $\mu_{-4, \beta}=-2/d$
in this case. 
If $-3/2d <k < -1/d$, the 
sequence 
$$0 \lr \oO_C(2) \lr E \lr I_C[1] \lr 0, $$
destabilizes $E\in P_4(X, \beta)$. Instead
the unique non-trivial extension 
$$ 0 \lr I_C[1] \lr E' \lr \oO_C(2) \lr 0, $$
becomes $\sigma$-limit stable. The object
$E'$ is isomorphic to a two term complex 
$\oO_X \stackrel{s}{\to} \oO_C(1)^{\oplus 2},$
and the sequence 
$$0 \lr \oO_C(1)^{\oplus 2} \lr E' \lr \oO_X[1] \lr 0, $$
destabilizes $E'$ if $k>-1/d$. 
We have $\Hom(\oO_C(1)^{\oplus 2}, \oO_X[2])=0$, and 
in fact $\lL_4^{\sigma}(X, \beta)=P_{-4}(X, \beta)=\emptyset$ for $k>-1/d$. 
The counting invariants are as follows, 
\begin{align*}
L_{4, \beta}(\sigma)= \left\{ \begin{array}{cl} 
4  & \mbox{ if }k<-\frac{3}{2d}, \\
1 & \mbox{ if }-\frac{3}{2d}<k < -\frac{1}{d}, \\
0 & \mbox{ if }k>-\frac{1}{d}.
\end{array} \right.
\end{align*}
For $P_{4, \beta}=4$, see~\cite[Section 4]{PT}. 
If $k_{-}<-3/2d$, $k_0=-3/2d$ and
$-3/2d<k_{+}<-1/d$, one can check that
(\ref{expect2}) holds.  
On the other hand, the formula (\ref{expect2}) 
is problematic if 
$$-\frac{3}{2d}<k_{-}<-\frac{1}{d}, \quad k_{0}=-\frac{1}{d}, 
\quad k_{+}>-\frac{1}{d}.$$ 
Since $\oO_C(1)^{\oplus 2}$ is not $\omega$-Gieseker
stable, we do not how to define $N_{4, 2[C]}$. 
 In this case, 
$N_{4, 2[C]}$ should be defined by Joyce's 
invariant~\cite[Definition 3.18]{Joy4}
 after involving virtual classes. 
For instance let us ignore virtual classes. By definition, 
Joyce's invariant
$N_{4, 2[C]}$
is the ``euler number'' of the following ``virtual'' stack, 
$$\left[\Spec \mathbb{C}/\GL(2, \mathbb{C})\right] -
\frac{1}{2}\left[\Spec \mathbb{C}/(\mathbb{A}^1 \rtimes
\mathbb{G}_m ^{2})\right], $$
which results $N_{4, 2[C]}=-1/4$. We have 
\begin{align*}
& L_{4, \beta}(\sigma_{-})-L_{4, \beta}(\sigma_{+})=1-0=1, \\
& (-1)^{4-1}4N_{4, 2[C]}L_{0, 0}(\sigma_0)=-1 \cdot 4\cdot (-1/4) \cdot 1=1,
\end{align*}
as desired.

Yukinobu Toda

Institute for the Physics and 
Mathematics of the Universe (IPMU), University of Tokyo, 

Kashiwano-ha 5-1-5, Kashiwa City, Chiba 277-8582, Japan

\textit{E-mail address}:toda@ms.u-tokyo.ac.jp, toda-914@pj9.so-net.ne.jp

\end{document}